\documentclass[11pt]{amsart}
\usepackage{graphicx, xcolor}
\usepackage{amssymb}
\numberwithin{equation}{section}

\def\N{\mathbb{N}}
\def\R{\mathbb{R}}

\renewcommand\d{\partial}

\def\g{\gamma}
\def\G{\Gamma}

\def\l{\lambda}

\def\epsilon{\varepsilon}
\def\e{\varepsilon}

\def\L{\Lambda}

\newcommand\br{\begin{rem}}
\newcommand\er{\end{rem}}
\newcommand\bp{\begin{pmatrix}}
\newcommand\ep{\end{pmatrix}}
\newcommand\be{\begin{equation}}
\newcommand\ee{\end{equation}}
\newcommand\ba{\begin{equation}\begin{aligned}}
\newcommand\ea{\end{aligned}\end{equation}}

\newtheorem{theorem}{Theorem}[section]
\newtheorem{proposition}[theorem]{Proposition}
\newtheorem{lemma}[theorem]{Lemma}
\newtheorem{corollary}[theorem]{Corollary}

\newtheorem{remark}[theorem]{Remark}
\newtheorem{ans}[theorem]{Definition}

\title{Metastable dynamics of internal interfaces for a convection-reaction-diffusion equation} 

\begin{document}

\maketitle

\begin{center}
MARTA STRANI\footnote{Universit\`a di Milano Bicocca, Dipartimento di Matematica e Applicazioni, Milano (Italy). E-mail address: \texttt{marta.strani@unimib.it}, \texttt{martastrani@gmail.com}.}
\end{center}
\vskip1cm

\begin{abstract}
We study a one dimensional metastable dynamics of internal interfaces for the initial boundary value problem for the following convection-reaction-diffusion equation 
\begin{equation*}
\d_t u = \e \d_x^2 u -\d_x f(u)+ f'(u).
\end{equation*}
A metastable behavior appears when the time-dependent solution develops into a layered function in a relatively short time, and subsequent approaches its steady state in a very long time interval.
A rigorous analysis is used to study such behavior, by means of the construction of a one-parameter  family $\{ U^\e(x;\xi)\}_\xi$ of approximate stationary solutions and of a linearization of the original system around an element of this family. We obtain a system consisting of an ODE for the parameter $\xi$, describing the position of the interface, coupled  with a PDE for the perturbation $v$, defined as the difference $v:=u-U^\varepsilon$. The key of our analysis are the spectral properties of the linearized operator around an element of the family $\{ U^\varepsilon \}$: the presence of a first eigenvalue, small with respect to $\varepsilon$, leads to a metastable behavior when $\varepsilon \ll 1$.

\end{abstract}

\begin{quote}\footnotesize\baselineskip 14pt 
{\bf Key words.} 
Metastability, slow motion, internal interfaces, reaction-convection-diffusion equations, spectral analysis.
 \vskip.15cm
\end{quote}

\begin{quote}\footnotesize\baselineskip 14pt 
{\bf AMS subject classification.} 
35K20, 35B36, 35B40, 35P15
 \vskip.15cm
\end{quote}

\pagestyle{myheadings}
\thispagestyle{plain}
\markboth{M.STRANI}{METASTABILITY FOR A CONVECTION-REACTION-DIFFUSION EQUATION}

\section{Introduction}

The slow motion of internal shock layers has been widely studied for a large class of evolutive PDEs on the form
\begin{equation*}
\d_t u = \mathcal P^\e[u],
\end{equation*}
where $\mathcal P^\e[u]$ is a nonlinear differential operator that  depends singularly on the parameter $\e$. Such phenomenon is known as {\it metastability}. The qualitative features of a metastable dynamics are the following: through a transient process, a pattern of internal layers is formed from initial data over a $\mathcal O(1)$ time interval; once this pattern is formed, the subsequent motion of such interfaces is exponentially slow, converging to their asymptotic limit. As a consequence, two different time scales emerge: for short times, the solutions are close to some non-stationary state; subsequently, they drift towards the equilibrium solution with a speed rate that is exponentially small.

In other words, the equation exhibits in finite time metastable shock profiles (called {\it interfaces}) that persist during an exponentially (with respect to a small parameter) long time period  and that move with exponentially slow speed.

\vskip0.2cm
Many fundamental partial differential equations, concerning different areas, exhibit such behavior. Among others, we include viscous shock problems (see, for example \cite{LafoOMal94}, \cite{LafoOMal95}, \cite{MS}, \cite{ReynWard95} for viscous conservation laws, and \cite{BerKamSiv01}, \cite{SunWard99} for Burgers type's equations), relaxation models as the Jin-Xin system \cite{Str12}, phase transition problems described by the Allen-Cahn equation, with the fundamental contributions \cite{CarrPego89, FuscHale89} and the most recent references \cite{OttoRezn06,Str13}, and the Cahn-Hilliard equation studied in \cite{AlikBateFusc91} and \cite{Pego89}.

In this paper we study the slow motion of internal interfaces generated by the evolution of the solution to a convection-reaction-diffusion equation. Given $\ell >0$, we consider the  initial-boundary value problem
\begin{equation}\label{ForcBurg}
 \left\{\begin{aligned}
		\partial_t u	& =\varepsilon\,\partial_x^2u - \partial_x f(u)+ f'(u),
		&\qquad &x\in I, t > 0,\\
 		u(0,t)	& =u(\ell,t)=0, &\qquad &t >0,\\
		u(x,0)		& =u_0(x), &\qquad &x\in I,
 	 \end{aligned}\right.
\end{equation}
where $I=[0,\ell]$ is a bounded interval of the real line, and the unknown $u \in C^0(\R^+;H^1(I))$. Here $\e$ is a small and positive parameter, that can be seen as a viscosity coefficient, while $f$  satisfies
\begin{equation}\label{ipof}
f''(u) >0, \quad f(0)=f'(0)=0, \quad f'(-u)=-f'(u).
\end{equation}
The main example we have in mind is the initial-boundary value problem for the generalized Burgers equation, also known as the Burgers-Sivashinsky equation,  that is
\begin{equation}\label{ex:burgers}
\partial_t u =\varepsilon\,\partial_x^2u - u \partial_x u+ u,
\end{equation}
together with boundary conditions and initial datum $u_0(x)$. Such equation arises from the study of the dynamics of an upwardly propagating flame-front in a vertical channel (see \cite{RakSiv87}); precisely, setting $u(x,t)=-\partial_x y(x,t)$,  the dimensionless shape $y=y(x,t)$ of the flame front interface satisfies 
\begin{equation}\label{SW2}
\left\{\begin{array}{ll}
\begin{aligned}
\partial_t y&=\frac{1}{2}(\partial_x y)^2+\varepsilon \partial_x^2 y+y-\int_0^\ell y \ dx, \quad &x \in (0,\ell), \quad t > 0, \\
\partial_x y(0, t)&=\partial_x y(\ell,t) =0, \quad & t >0, \\
y(x,0)&=y_0(x), \quad & x \in (0,\ell).
\end{aligned}
\end{array}\right.
\end{equation}

\vskip0.2cm
We mean to analyze the behavior of solutions $u^\e$ to \eqref{ForcBurg} in the vanishing viscosity limit, i.e. $\varepsilon\to 0$. In particular, we address the question whether a phenomenon of metastability occurs for such kind of problem, and how this special dynamics is related to the size of the viscosity coefficient and to the choice of the initial datum $u_0$.  
\vskip0.2cm
In the limit $\e\to 0$, equation \eqref{ForcBurg} formally reduces to the first order hyperbolic equation
\begin{equation}\label{burgforchyp}
\d_t u= - \d_xf(u)+ f'(u), \quad u(x,0)=u_0(x).
\end{equation}
The set of solutions for such equation is the one  given by the entropy formulation, in the sense of Kruzkov (see \cite{Kru70}). In this case it is well know the existence and uniqueness of the solution $U_{{}_{\rm hyp}}(x,t)$ to the Cauchy problem for \eqref{burgforchyp} (i.e. $x \in \R$). Moreover, denoting by $u^\e$ the solution to \eqref{ForcBurg}, it is possible to prove that  
\begin{equation*}
 u^\e(x,t) \to U_{{}_{\rm hyp}}(x,t) \ \ \ {\rm as} \ \ \ \ \e \to 0,
 \end{equation*}
 (for more details, see \cite{Lax54,Lyb91,Ole63}). Finally, such convergence  is in $L^1_{{}_{\rm loc}}$, and it is uniform away from shock waves, as proven in \cite{Nes96}.
 \vskip0.2cm
 When Dirichlet boundary conditions are taken into account, the case of a bounded domain is more delicate than the Cauchy problem. First of all, the boundary conditions $u(0,t)=u(\ell,t)=0$ has to be interpreted in a nonclassical way in the sense of \cite{BardLeRoNede79}. 

Concerning existence results and asymptotic behavior for the solutions to the boundary-value-problem for  \eqref{burgforchyp}, a complete analysis  has been performed by C. Mascia and A. Terracina in \cite{MascTerr99}; here the authors deal  with a general reaction-convection equation of the form
\begin{equation}\label{RCintro}
 \left\{\begin{aligned}
		\partial_t u	& = - \partial_x f(u)+ g(x,u),
		&\qquad &x\in (0,\ell), t \geq 0,\\
 		u(0,t)	& =u_-, \ \ u(\ell,t)=u_+, &\qquad &t \geq0,\\
		u(x,0)		& =u_0(x), &\qquad &x\in (0,\ell),
 	 \end{aligned}\right.
\end{equation}
showing how the boundary conditions influence the large-time behavior of solutions. Moreover, they prove the presence of discontinuous stationary solutions for \eqref{RCintro}, that correspond to stationary solutions with internal layers for the viscous problem when $\e >0$.

\vskip0.2cm
When $\e >0$, the presence of a diffusive term has a smoothing effect on solutions. Indeed, all the discontinuities turn into smooth internal layers, that are less sharp as $\e$ increases. 

Numerical computations show that, for a certain class of initial data, equation \eqref{ex:burgers} exhibits a metastable behavior. Precisely, in \cite{BerKamSiv95}, it has been proved that a sufficient condition for the appearance of a metastable dynamics is that the initial datum satisfies
\begin{equation}\label{u0meta}
u_0(x) <0 \ \ {\rm in} \ \ (0,a_0), \quad u_0(x) >0 \ \ {\rm in} \ \ (a_0, \ell) \ \ \ {\rm for \ some } \ \ a_0 \in (0,\ell),
\end{equation} 
where $u_0(x)$ has to be consider in $ C^0_0(I)$.

Starting from such initial configuration, one observe that an interface is formed in an $\mathcal O(1)$ time scale. Once such interface is formed, it starts to move towards one of the wall $x=0$ or $x=\ell$, but this motion is extremely slow (see Figure \ref{fig0}). Hence, two different time scales emerge: a first transient phase of order $\mathcal O(1)$ in time where the interface is formed, and an exponentially long time phase, that can be extremely long provided $\e$ very small, where the interface drifts towards its equilibrium configuration. 

\begin{figure}
\centering
\includegraphics[width=14cm,height=10cm]{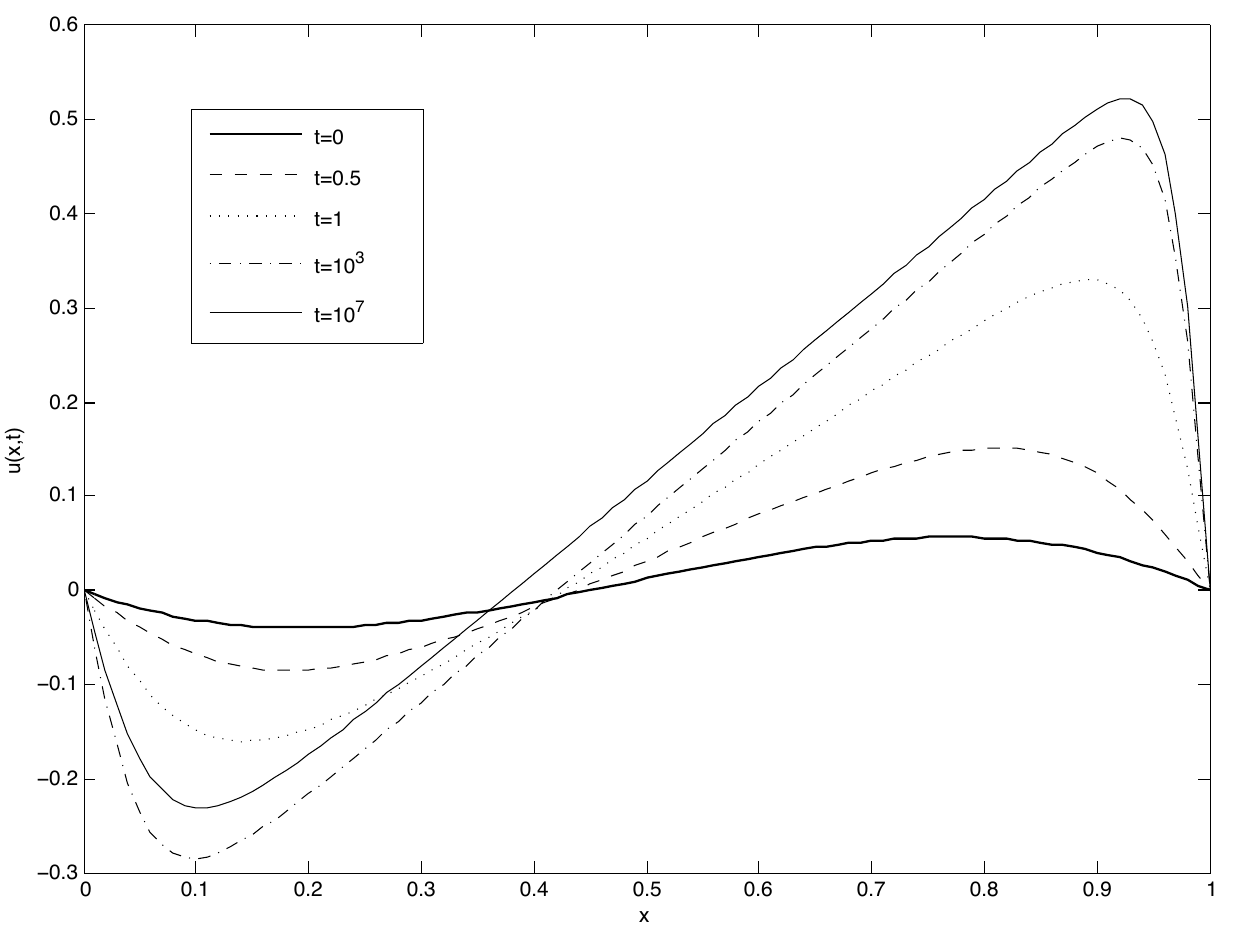}
\caption{\small{The evolution of the solution to \eqref{ForcBurg} with $f(u)=u^2/2$ and $\varepsilon=0.01$. Once the interface is formed in an order $1$ time scale, we observe that it starts to move towards its equilibrium configuration, corresponding in this case to the wall $x=0$. As we can see from the picture, this motion is extremely slow. }}\label{fig0}
\end{figure}

In terms of the shape $y(x,t)$, in a first stage of its dynamics  the solution assumes a somewhat asymmetric parabolic shape. In particular, the tip of such parabola corresponds to the point where $u$ vanishes. The subsequent motion of the tip of the parabolic flame-front interface towards one of the wall $x=0$ or $x=\ell$ can be extremely slow with respect to the parameter $\e$.
\vskip0.2cm

There are several papers concerning the dynamics of solutions to equation \eqref{ForcBurg}, that is  a special case of a more general class of convection-reaction-diffusion equations on the form
\begin{equation}\label{CRDeq}
\d_t u = \e \d_x^2 u- f(x,u) \d_x u + g(x,u).
\end{equation}
The initial-boundary value problem in different dimensional sets  for \eqref{CRDeq} has been investigated in a numbers of works, under different assumptions for the functions $f$ and $g$. To name some of these papers, we recall here \cite{BobOreRoy87, CheLevSac88, HillSul95, How88, HowWhi88}.

 \vskip0.2cm
A pioneering article in the study of the metastable dynamics of solutions for the convection-reaction-diffusion equation \eqref{ForcBurg} is the reference \cite{BerKamSiv01}. Here the authors analyze the specific case $f(u)=u^2/2$ (Burgers-Sivashinsky equation), proving that there exist different types of stationary solutions: more precisely, there exist one positive stationary solution $U^+_\e$ and one negative stationary solution $U^-_\e$ that are linearly stable, and two other unstable equilibrium solutions $U^\pm_{\e,1}$, that have exactly one zero inside the interval. 

Moreover, it turns out that the stationary solution $U^-_{\e,1}$ (that, in terms of the shape $y$, corresponds to a parabolic-shaped flame-front interface) is {\bf metastable}. Indeed, concerning the time-dependent problem, in \cite{BerKamSiv01}  the authors rigorously prove that, for the particular class of initial data considered  in \cite{BerKamSiv95}, the solutions generated by such initial configurations exhibit a metastable behavior, in the sense that they remain close to the initial configuration for a time of order $T_\e=e^{1/\e}$, before converging to one of the stable steady states $U^\pm_\e$.

Such slow motion is a consequence of the presence of a first small eigenvalue associated with  the linearization around the unstable stationary solution $U^-_{\e,1}$: in particular, $\lambda_1$ is positive but exponentially small in $\e$. Therefore, starting from initial data that are small perturbation of such unstable steady state, the corresponding time-dependent solution starts to move towards one of the  stable equilibrium configurations $U^\pm_\e$, but this motion is extremely slow, since it is described by terms of order $e^{\lambda_1^\e t}$.

Hence, differently from the cases considered in \cite{MS,MS2,Str12}, here the metastable behavior is characterized by the fact that the steady state $U^-_{\e,1}$ is unstable, so that the solutions starting from an initial configuration close to $U^-_{\e,1}$ are pushed away towards their asymptotic limit, but the time of convergence can be extremely long.

\vskip0.2cm

The problem of slow motion for equation \eqref{ForcBurg} has been examined  also in \cite{SunWard99} for a generic flux function $f(u)$ that satisfies hypotheses \eqref{ipof}. Here the authors show that the first eigenvalue associated with the linearized problem around the equilibrium solution $U^-_{\e,1}$ is exponentially small with respect to $\e$; more precisely,  they provide an asymptotic expression for such principal eigenvalue. The presence of a first small eigenvalue leads to a metastable behavior for the time-dependent problem, studied by using the so-called projection method:  the authors are able to derive an (asymptotic) ordinary differential equation describing the motion of the unique zero of the solution $u$, that corresponds, in term of the shape $y$, to the tip location of a parabolic-shaped interface.
\vskip0.2cm
The existence and the stability properties of stationary solutions to \eqref{ForcBurg} that have more than one zero inside the interval, and the corresponding time-dependent problem, has been conversely studied in \cite{Goo94}.

\vskip0.2cm
The aim of this paper is to prove existence results for the stationary solutions to \eqref{ForcBurg} in the case of a generic flux function $f(u)$, and to rigorously study the subsequent motion of solutions starting from initial data on the form \eqref{u0meta}. 
\vskip0.2cm
The first part of our study thus concerns the stationary problem
\begin{equation}\label{stazprobintro}
\varepsilon\,\partial_x^2u = \partial_x f(u)- f'(u),
\end{equation}
complemented with boundary conditions $u(0)=u(\ell)=0$. Our first main result gives a  description of the solutions to \eqref{stazprobintro}.

\begin{theorem}
There exist one positive solution $U_{\e,+}$ and one negative solution $U_{\e,-}$ to \eqref{stazprobintro}. Additionally,
\begin{itemize}
\item For $\e \to 0$, $U_{\e,+}$ converges pointwise to $x$ in $(0,\ell)$. 
\item For $\e \to 0$, $U_{\e,-}$ converges pointwise to $x-\ell$ in $(0,\ell)$. 
\end{itemize}
Moreover, there exist two other solutions $U^\e_{{}_{M}}$ and $U^\e_{{}_{NS}}$ that have one zero inside the interval $(0,\ell)$ and such that
\begin{itemize}
\item For $\e \to 0$, $U^\e_{{}_{M}}$ converges pointwise to $x-\ell/2$ in $(0,\ell)$. 
\item For $\e \to 0$, $U^{\e}_{{}_{NS}}$ converges pointwise to the function 
\begin{equation*}
U^0_{{}_{NS}}(x):= \left\{ \begin{aligned}
& x \quad & x \in (0,\ell/2), \\
& x-\ell \quad & x \in (\ell/2,\ell),
\end{aligned}\right.
\end{equation*}
for $x\neq \ell/2$. 
\end{itemize}
\end{theorem}

Concerning the long-time dynamics of the solutions to the time-dependent problem \eqref{ForcBurg}, the strategy we mean to use here is analogous to the one firstly performed in \cite{MS} to study the slow motion of internal layers for parabolic evolutive systems, under appropriate assumptions on the spectrum of the linearized operator around the steady state. In particular, in \cite{MS} it was required that the spectrum of such linearized operator was composed by real and negative eigenvalues (for more details, see \cite[Section 2]{MS} ). For the problem studied in this paper, we know that there exists a first positive eigenvalue that is exponentially small in $\e$, so that the stationary solution is unstable but it is metastable, as already stressed before.  Hence, we need to slightly modify the assumptions we have to require. This is the strategy we mean to follow.
\vskip0.1cm
$\bullet$ We ask for the existence of a one parameter family of functions $\{ U^\e(x;\xi)\}_{\xi \in I}$, where the parameter $\xi$ represents the unique zero of the function $U^\e$, i.e. the location of the interface, and such that each element of the family can be seen as an approximation of the unstable steady state $U^\e_{{}_{M}}$ in a sense that we will specify later.

\vskip0.1cm
$\bullet$ We linearize around an element of the family, by looking for a  solution to \eqref{ForcBurg} in the form $u(x,t)= U^\e(x;\xi(t))+ v(x,t)$.

\vskip0.1cm
$\bullet$ We use a modified version of the projection method in order to obtain a coupled system for the perturbation $v$ and for the interface location $\xi$.

\vskip0.1cm
$\bullet$ We prove that, under appropriate assumptions on the spectrum of the linearized operator obtained from the linearization around the steady state and on the family of functions $\{U^\e\}$, a metastable behavior occurs.

\vskip0.2cm
This strategy dates back the work of Carr and Pego \cite{CarrPego89}: here the authors consider a function $u^{\xi}(x)$  (with $\xi=(\xi_1, \, ......\, , \xi_N)$ describing the position of the interfaces) which approximates a metastable state with $N$ transition layers. The admissible layer positions lie in a set $\Omega_{\rho}$, where $\xi_j-\xi_{j-1}>\varepsilon/\rho$, so that the set of states $u^\xi$ forms an $N$-dimensional manifold $\mathcal M=\{ u^\xi: \xi \in \Omega_{\rho}\}$. To study the dynamics of solutions located near $\mathcal M$, the authors linearize around an element of the manifold and study spectral properties of the linearized operator.

Starting from there, the strategy of constructing invariant manifolds of approximate steady states have been widely used to describe the slow motion of solutions to different PDEs. We quote here  the reference \cite{BeckWayn09}, as well as the more recent contributions \cite{MS,Str12,Str13}.

The projection method has been also employed in \cite{SunWard99}: here the authors utilize such approach in order to give an explicit asymptotic characterization of the metastable motion of the interface, by deriving an
{\it asymptotic} ordinary differential equation describing its slow motion (see \cite[Section 4]{SunWard99}).
\vskip0.2cm
In our framework, the main difference with respect to \cite{SunWard99} is that here we derive a {\it rigorous} ODE for the position of the interface, where also the nonlinear terms are taken into account; these terms keep track of the nonlinear evolution of the variable $\xi$, when starting far away from its equilibrium configuration.

Also, the projection method allows us to derive an explicit PDE for the perturbation $v$; in order to use a spectral decomposition for  $v$, we exploit the spectral properties of the linearized operator and we achieve a rigorous result by deeply using the structure of the equation.

\vskip0.2cm

Going deeper in  details, these are the objects we shall use in the following and the hypotheses we need to state.
\vskip0.2cm
{\bf H1. Hypotheses on the family of approximate steady states} 
\vskip0.2cm
There exists a  family  $\{ U^\e(x;\xi) \}_{\xi \in I}$  such that

\vskip0.2cm

{\bf i)} There exists a value $\bar \xi \in I$ such that the element $U^\e(x;\bar \xi)$ corresponds to a stable 

\quad \, \ steady state for the original equation.

{\bf ii)} For every $\xi \in I$, each element of the family satisfies
\begin{equation*}
U^\e(x;\xi) >0 \ \ \ {\rm for} \ \ \ x \in (0,\xi) \quad {\rm and} \quad  U^\e(x;\xi) <0 \ \ \ {\rm for} \ \ \ x \in (\xi,\ell).
\end{equation*}

{\bf iii)} There exists a family of smooth and positive functions $\Omega^\e(\xi)$, uniformly converging 

\quad \, \,  to zero as $\e \to 0$, such that there holds
\begin{equation*}
|\langle \psi(\cdot ), \mathcal P^\e[U^\e(\cdot;\xi)]\rangle| \leq \Omega^\e(\xi)|\psi|_{L^\infty}, \quad \forall \, \psi \in C(I), \, \forall \, \xi \in I.
\end{equation*}

{\bf iv)} There exists a family of smooth positive 
functions $\omega^\varepsilon=\omega^\varepsilon(\xi)$, uniformly convergent 

\quad \, \ to zero as 
$\varepsilon\to 0$,  such that
\begin{equation*}
	\Omega^\varepsilon(\xi)\leq \omega^\varepsilon(\xi)\,|\xi-\bar \xi|.
\end{equation*}

\vskip0.2cm
{\bf H2. Hypothesis on the eigenvalues and on the eigenfunctions of the linearized operator} 

\vskip0.2cm
Let  $\mathcal L^\e_\xi$ be the linearized operator obtained from the linearization of the equation \eqref{ForcBurg} around an element of the family $\{ U^\e(x;\xi) \}$. Also, let  $\{ \lambda^\e_k(\xi)\}_{k \in \N}$ be the sequence of the eigenvalues of $\mathcal L^\e_\xi$, and let $\varphi^\varepsilon_k(\cdot;\xi)$ and $\psi^\varepsilon_k(\cdot;\xi)$ be the eigenfunctions of $\mathcal L^\e_\xi$ and its adjoint $\mathcal L^{\e,*}_\xi$ respectively.

\vskip0.2cm

{\bf i)} The sequence of eigenvalues $\{ \lambda^\e_k(\xi)\}_{k \in \N}$ is such that 
\vskip0.2cm
\begin{itemize}
\item  $\l^\e_1(\xi) \to 0$ as $\e \to 0$ uniformly with respect to $\xi$.
\vskip0.1cm
\item  All the eigenvalues $\{ \l^\e_k \}_{k\geq 2}$ are negative and there exist constants $C,C'$ such that
\begin{equation*}
\lambda^\e_1(\xi)-\lambda_2^\e(\xi) \geq C' \ \ \  \forall \ \xi \in I, \qquad \l_k^\e(\xi) \leq -C k^2, \ \ \ {\rm if} \ \ k \geq 2.
\end{equation*}

\end{itemize}

\vskip0.2cm
{\bf ii)} The eigenfunctions $\varphi^\varepsilon_k(\cdot;\xi)$ and $\psi^\varepsilon_k(\cdot;\xi)$
are normalized so that
\begin{equation*}
	\langle \psi^\varepsilon_1(\cdot;\xi), \partial_{\xi}U^{\varepsilon}(\cdot;\xi)\rangle=1
	\qquad\textrm{and}\qquad
	\langle \psi^\varepsilon_j, \varphi^\varepsilon_k\rangle
	=\left\{\begin{aligned} & 1	&\qquad &\textrm{if }j=k,\\ & 0	&\qquad &\textrm{if }j\neq k,
		\end{aligned}\right.
\end{equation*}
\quad \ \ \  \ and we assume 
\begin{equation*}
	\sum_{j} \langle \partial_\xi \psi^\varepsilon_k, \varphi^\varepsilon_j\rangle^2
	=\sum_{j} \langle \psi^\varepsilon_k, \partial_\xi \varphi^\varepsilon_j\rangle^2
	\leq C	\qquad\qquad\forall\,k.
\end{equation*}
\quad \ \ \  \  for some constant $C$ that does not depend on $\xi$.

\begin{remark} { \rm

Hypothesis {\bf H1} concerns the behavior of the family $\{ U^\varepsilon\}$ in the limit $\varepsilon \to 0$. Because of these assumptions, each element of the family can be seen as an approximation for the exact steady state of the equation. In particular, the term $\Omega^\varepsilon$ measures how far is an element of the family from being an exact steady state; indeed, $\mathcal P^\varepsilon[\bar u]$ is equal to zero when $\bar u$ solves the stationary equation. We also point out that the existence of this family is not guaranteed, and its expression depends on the specific form of the function $f$; as we will see in section 5 for the special case $f(u)=u^2/2$, the functions $U^\e$ have to be constructed ``by hands", and it has to be checked that the hypotheses required are satisfied.

Hypothesis ${\bf H2}_{{\bf i}}$ is a crucial hypothesis regarding the spectral properties of the linearized operator when linearizing around an  approximate steady state $U^\varepsilon$. The assumptions we asked are however justified by the the behavior of the exact steady states solutions to \eqref{ForcBurg}. Also, we will show in section 5 that these hypotheses are satisfied in the special case of the Burger-Sivashinsky equation, i.e. $f(u)=u^2/2$. 

Finally, hypothesis ${\bf H2}_{{\bf ii}}$ is a technical hypothesis needed in order to prove our main Theorems (for the details, see Section 4).

}
\end{remark}

\vskip0.3cm Under these hypotheses, as stated above, we are able to prove that the solution to the initial-bounday-value problem \eqref{ForcBurg} experiences a metastable behavior. Precisely, we  describe the slow motion of the solution  by describing the behavior of the perturbation $v$ and of the interface location $\xi(t)$; at first, we consider a simplified partial differential equation for the perturbation $v$, where the higher order terms arising from the linearization are canceled out.

Our first main contribution is the following

\begin{theorem}\label{teointro}

Let $u(x,t)=v(x,t)+ U^{\varepsilon}(x;\xi(t))$ be the solution of the initial-boundary value problem \eqref{ForcBurg}. Let assume that hypotheses {\bf H1-H2} are satisfied. Hence, for $\varepsilon$
  sufficiently small, there exists a time $T^{\varepsilon}$ of order $e^{1/\varepsilon}$ such that, for $t \leq T^\varepsilon$ the following bounds hold
\begin{equation*}
|v|_{{}_{L^2}}\,
		\leq c_1 |v_0|^2_{{}_{L^2}} e^{-c t} + c_2\,|\Omega^\varepsilon|_{{}_{L^\infty}}, \quad {\rm and} \quad |\xi(t) - \bar \xi | \leq |\xi_0| e^{-\beta^\varepsilon t}, \quad \beta^\varepsilon \to 0 \ {\rm as} \ \varepsilon \to 0.
\end{equation*}
\end{theorem}

\begin{remark}  \rm{
Theorem \ref{teointro} states that the perturbation $v$ has a very fast decay in time up to a reminder that behaves like $ e^{-1/\varepsilon}$, so that the solution $u$ to \eqref{ForcBurg} is drifting to its equilibrium configuration at a speed rate dictated by $\beta^\varepsilon$; hence, the convergence towards the steady state is much slower a $\varepsilon$ becomes smaller.}

\end{remark}

Subsequently, we consider the complete system for the perturbation $v$, where also the higher order terms are taken into account. This leads to the second main contribution of this paper

\begin{theorem}\label{teointro2}
Let $u(x,t)=v(x,t)+ U^{\varepsilon}(x;\xi(t))$ be the solution of the initial-boundary value problem \eqref{ForcBurg}. Let assume that hypotheses {\bf H1-H2} are satisfied. Then, for $\varepsilon$
  sufficiently small, there exists a time $T^{\varepsilon}$ of order $1/\varepsilon^{\alpha}$ for some $\alpha \in (0,1) $ such that, for $t \leq T^\varepsilon$ the following bounds hold
\begin{equation*}
|v|_{{}_{H^1}}\,
		\leq c_1 |v_0|^2_{{}_{H^1}} e^{-c t} + c_2 \varepsilon^\delta, \quad {\rm and} \quad |\xi(t) - \bar \xi | \leq |\xi_0| e^{-\beta^\varepsilon t}, \quad -\beta^\varepsilon \to 0 \ {\rm as} \ \varepsilon \to 0,
\end{equation*}
where $\delta \in (0,1)$.

\end{theorem}

\begin{remark} \rm{
As we will see in details in the next sections, the nonlinear terms in the equation for the perturbation $v$ depend also on the first space derivative, so that, in order to prove Theorem \ref{teointro2}, an additional bound for the $H^1$-norm of $v$ is needed.  This is the reason why the final bound for $v$ is weaker than the corresponding formula  stated in Theorem \ref{teointro}. Also, the final time $T^\varepsilon$ si diverging to infinity as $\varepsilon^{-\alpha}$, $\alpha \in (0,1)$, rather than $e^{1/\varepsilon}$.  }
\end{remark}

As a direct consequence of Theorem \ref{teointro} and \ref{teointro2}, we can state the following Corollary concerning the solution to the initial-boundary-value problem \eqref{ForcBurg}.

\begin{corollary}
Let $u(x,t)$ be the solution to \eqref{ForcBurg}, with initial datum $u_0$ on the form \eqref{u0meta}. If $a_0 \in (0,\ell/2)$, then $u(x,t)$ converges to $U_{\e,+}(x)$ for $t\to +\infty$.
Conversely, if $a_0 \in (\ell/2,\ell)$, then $u(x,t)$ converges to $U_{\e,-}(x)$ for $t\to +\infty$. In both cases,  the speed rate of convergence is given by $\beta^\varepsilon$, and the time of convergence is proportional to $T^\varepsilon$.

\end{corollary}

\vskip0.1cm
The main difference with respect to the other papers that have considered the problem of metastability for equation \eqref{ForcBurg} is that here we deal with a generic flux function $f$ that satisfies hypotheses \eqref{ipof}; also, we develop a general theory to rigorously prove the slow motion of the internal interfaces, that could be also applicable to other types of convection-reaction-diffusion equations and, hopefully, to the case of systems, provided that the assumptions {\bf H1-H2} are satisfied. As we shall see in details in Section 4, the proofs of Theorem \ref{teointro} and \ref{teointro2} can be extended to the case of an unknown $v \in [L^2(I)]^n$, $n\geq 2$, with only minor changes (see also \cite[Theorem 2.1]{MS}). In this direction, we quote here the recent contributions \cite{MS2} and \cite{Str12} , where the the isentropic Navier-Stokes system and hyperbolic-parabolic Jin-Xin system have been considered.  In principle, when rigorous results are not achievable, it could be possible to obtain numerical evidence of the spectrum of the linearized operator.

Moreover, in this paper we are able to give explicit expressions for the speed and for the size of the interface location $\xi$, as well as for the time of convergence of such interface towards its equilibrium configuration; this is a direct consequence of the strategy we used, that is the description of the solution to \eqref{ForcBurg} as the sum of two functions, each of them satisfying an explicit equation. As a consequence, the two phases of the dynamics are explicitly described and separated.

\vskip0.5cm
We close this introduction with an overview of the paper. 

In section 2 we prove the existence of four different types of stationary solutions for the equation \eqref{ForcBurg}, and we discuss the stability properties of these steady states. 

In section 3 we develop a general approach to describe the dynamics of solutions belonging to a neighborhood of a one-parameter family $\{ U^\e(x;\xi) \}_{\xi \in I}$ of approximate steady states, where we use as coordinates the parameter $\xi$, describing the location of the internal interface, and the perturbation $v$, describing the distance between the solution $u$ and an element of the family. By linearizing the original equation around an element of the family, we end up with a coupled system for the variables $(\xi,v)$, whose analysis is performed in the subsequent Section 4. In particular we here deal with an approximation of the system, obtained by linearizing with respect to $v$ and by disregarding the $o(v) -$terms. 
Specifically, we state and prove Theorem \ref{thm:metaL}, providing, under appropriate assumptions on the spectrum of the linearized operator around $U^\e$ as well as on the behavior of $U^\e$ as $\e \to 0$, an explicit estimate for the perturbation $v$. Such estimate will be subsequently used to decoupled the system for $(\xi, v)$, in order to obtain a reduced equation for the interface location $\xi(t)$, analyzed in Proposition \ref{prop:slowmotion}. In particular, these results characterizing the couple $(\xi,v)$ give a good qualitative explanation of the transition from the metastable state to the finale stable state.

 Last part of this section is devoted to the analysis of the complete system for the couple $(\xi, v)$, where also the higher order terms in $v$ are considered: the main contribution of this section is Theorem \ref{thm:metaNL}, where we prove an estimate for the difference $|(v-z)(t)|_{{}_{H^1}}$,where $z$ is a function with a very fast decay in time. 
This result, together with Theorem \ref{thm:metaL}, makes the theory complete.

Finally, in section 5, we consider, as an example, the Burgers-Sivashinsky equation: in this  case we are able to provide an explicit expression for the approximated family $\{ U^\e\}$. In order to apply the general theory developed in the previous sections, we give a measure on how far is an element of the family $\{ U^\e\}$ from being an exact steady state, as well as an explicit expression for the speed of convergence of the interface. It turns out that all these terms are small with respect to $\varepsilon$. Subsequently, we analyze spectral properties of the linear operator arising from the linearization around the approximate steady state $U^\e$, showing that the spectrum can be decomposed as follows: the first eigenvalue $\lambda_1^\e$ is positive and of order $e^{-1/\e}$; all the remaining eigenvalues $\{ \lambda_k^\e\}_{{}_{k\geq 2}}$ are negative and behave like $-C/\sqrt{\e}$. Such estimates will translate into a one-dimensional dynamics, since  all of the components of the perturbation relative of all the eigenvectors except the first one will have a very fast decay for $\varepsilon$ small, and in a slow motion for the interface as a consequence of the size estimate for the first eigenvalue. This analysis is needed to give evidence of the validity of the assumptions of Theorem \ref{thm:metaL} and Theorem \ref{thm:metaNL}.

\section{The stationary problem}

In this Section we deal with the stationary problem for \eqref{ForcBurg}, that is
\begin{equation}\label{ForcBurgStat}
 \left\{\begin{aligned}
		\varepsilon\,\partial_x^2u &= \partial_x f(u)- f'(u)\\
 		u(0)	& =u(\ell)=0	 \\
 	 \end{aligned}\right.
\end{equation}
for $x\in (0,\ell)$. This problem has been extensively studied in the case $f(u)=u^2/2$ in the work of  H. Berestycki, S. Kamin and G. Sivanshinsky \cite{BerKamSiv01}. Here the authors prove the existence and uniqueness of four type of solutions to \eqref{ForcBurgStat}: they prove that there exist a unique positive solution $U^+_\e$, a unique negative solution $U^-_\e$, and two other stationary solutions $U^+_{1,\e}$ and $U_{1,\e}^-$, which have one zero inside the interval (for more details, see \cite[Theorem 1]{BerKamSiv01}). Additionally, the authors prove that the solutions $U^{\pm}_\e$ are stable, while $U^\pm_{\e,1}$ are unstable with respect to perturbations of initial data (see \cite[Theorem 6.4]{BerKamSiv01}).

\vskip0.2cm
\subsection{Existence of stationary solutions} We here mean to use analogous techniques to those used in \cite{BerKamSiv01} in order to prove the existence of stationary solutions in the case of a generic flux function that satisfies hypotheses \eqref{ipof}. In particular we are interested in studying the existence of the stationary solution, named here $U^{\e}_{{}_{M}}(x)$, that gives rise to a metastable behavior.

Let us stress again that, in this case, $U^\e_{{}_{M}}(x)$ is said to be metastable because, starting from an initial datum located {\bf near} $U^\e_{{}_{M}}$, the solution drifts apart the unstable steady state towards one of the stable equilibrium configurations, and this motion is extremely slow. This behavior is different from other cases (see, for example, \cite{MS,Str12}) where the unique steady state is metastable in the sense that, starting from an initial configuration located {\bf far} from the equilibrium, the time-dependent solution starts to drifts in an exponentially long time towards the asymptotic limit.

To prove the existence of the metastable steady state, we first prove the existence of a positive steady state, named here $U_{\varepsilon,+}(x)$, and then we  deduce existence and properties of the steady state $U^\varepsilon_{{}_{M}}(x)$ by making use of symmetries and scalings in the problem.

Before stating our result, let us define the tools we shall use in the following.

\begin{ans}
A function $v \in H^1[0,\ell]$ is a subsolution (respectively a supersolution) to \eqref{ForcBurgStat} if $v(0), v(\ell) \leq 0$ (respectively $v(0),v(\ell) \geq 0$) and
\begin{equation*}
\int_0^\ell \left( \e v' \varphi' - f(v) \varphi ' -f'(v) \varphi \right) \, dx \leq 0 \quad ({\rm respectively} \ \geq0),
\end{equation*}
for all $\varphi \in C^1[0,\ell]$ such that $\varphi(0)=\varphi(\ell)=0$.

\end{ans}

\begin{remark}{\rm
In the following, we will prove existence of the solution starting from the existence of a subsolution and a supersolution and  by making use of a standard monotone iteration technique (see \cite{Evans, Sat72}).
}
\end{remark}

\begin{proposition}\label{propstazpos}
There exists a unique solution $U_{\e,+}(x)$ to  \eqref{ForcBurgStat}, that is positive in the interval $(0,\ell)$ and such that
\begin{equation*}
 U'_{\e,+} \leq 1, \quad  0< U_{\e,+} \leq x \quad {\rm and} \quad U''_{\e,+} \leq 0 \ \ {\rm for} \ \ x \in (0,\ell).
\end{equation*}

\end{proposition}

\begin{proof}

Denoting by $Nu:= -\e \d_x^2 u+ f'(u)\d_x u-f'(u)$ , the function $v(x)=x$ is such that $N v \geq 0$, that is $v$ is a supersolution. On the other side, given $\alpha \in\R^+$, we consider the function $v(x)=\alpha \sin \left( \frac{\pi}{\ell} \, x\right)$. We get
\begin{equation}\label{subsol}
\begin{aligned}
N v &= \e \alpha \left(\frac{\pi}{\ell} \right)^2 \sin \left( \frac{\pi}{\ell} \, x\right)+ f'(v) \left[  \frac{\alpha\,\pi}{\ell}\cos \left( \frac{\pi}{\ell}\, x\right)-1\right] \\
& \leq \e \alpha \left(\frac{\pi}{\ell} \right)^2 + f'(v) \left[\frac{\alpha\,\pi}{\ell}-1\right].
\end{aligned}
\end{equation}
Since $f'(v)$ is positive inside the interval $(0, \ell)$,  if we denote by $m=\max\limits_{x \in[0,\ell]} f'(v(x))$, in order to have $N v \leq 0$,  we have to require $\frac{\alpha\,\pi}{\ell} \leq 1$. Hence, we can choose $\alpha=\alpha(m)$ such that \eqref{subsol} is nonpositive, that is $v$ is a subsolution to \eqref{ForcBurgStat} in the interval $(0,\ell)$. Finally, since $\frac{\alpha\,\pi}{\ell} \leq 1$, we have
\begin{equation*}
\alpha\sin \left( \frac{\pi}{\ell} \, x\right) \leq x,
\end{equation*}
so that there exists a positive solution $u_+(x)$ to \eqref{ForcBurgStat} in $(0,\ell)$. 

To prove the uniqueness of the positive steady state, we only give a sketch of the proof, since the computations are similar to the ones used in \cite[Section 4]{BerKamSiv01} in the case $f(u)=u^2/2$. The idea is to rescale the problem by performing the change of variable $u(x)={\e}^\beta\, v\!\left(\frac{x}{{\e}^\beta}\right)$, for some $\beta>0$ chosen such that we get the following equation
\begin{equation}\label{rescaled}\partial^2_{x}v- f'(v)\partial_x v+f'(v)=0, \quad v(0)=v({\e}^\beta\,\ell)=0.\end{equation}
For example, if $f(u)=u^\gamma/\gamma$, then $\beta=1/\gamma$.
Uniqueness for the solution to \eqref{ForcBurgStat} corresponds to uniqueness for the solution to \eqref{rescaled}. We then introduce the following initial value problem
\begin{equation}\label{IBValpha}
\d_x^2v- f'(v)\d_x v+f'(v)=0, \quad v(0)=0, \ \ v'(0)=\alpha,
\end{equation}
for some $\alpha \in (0,1)$ and one can see that the solution $v_{\alpha}$ to \eqref{IBValpha} has a first zero, denoted here by $x=z(\alpha)$,  such that $v_\alpha >0$ in $(0,z(\alpha))$ and $v_\alpha(0)=v_\alpha(z(\alpha))=0$. To prove uniqueness for \eqref{rescaled} it is then sufficient to prove that $z(\alpha)$ is strictly increasing with respect $\alpha$. For the proof of this statement and for further details, we refer to \cite{BerKamSiv01}, Proposition 4.2 and Appendix A.

For the proof of the second part of the Proposition, we know that $u_+ (x) \leq x$ for all $x \in (0,\ell)$, so that $u_+'(0) \leq 1$; moreover, since $u_+>0$ and $u^+(\ell)=0$, there follows  $u_+'(\ell) <0$. Hence, there exists a maximum $x_1 \in (0,\ell)$ for the function $u^+$.
Now let us suppose that there exists a value $x_2 > x_1$ such that $u_+(x_2)$ is an internal minimum for $u_+$. From the equation we have
\begin{equation*}
\e \d_x^2 u_+(x_2)= -f'(u_+(x_2)) <0,
\end{equation*}
which is impossible since $x_2$ is a minimum for $u$. Hence, $u''_+(x)$ is negative for all $x \in (0,\ell)$. Finally, from the equation
\begin{equation*}
0 >\e \d^2_x u_+ = f'(u_+) (\d_x u_+ -1),
\end{equation*}
that is, since $f'(u_+)>0$, $\d_x u_+ < 1$.

\end{proof}

Starting from the existence of the function $U_{\e,+}$, the following results concerning $U^\e_M$ can be proved.

\begin{proposition}\label{prop:exstat}
There exists a unique  $U^{\e}_{{}_{M}}(x)$, solution to \eqref{ForcBurgStat}, and there exists  $x_0 \in (0,\ell)$ such that $U^{\e}_{{}_{M}}(x_0)=0$ and
\begin{equation*}
 {U^{\e}_{{}_{M}}}(x) <0 \ \ {\rm for} \ \  x < x_0, \quad {U^{\e}_{{}_{M}}}(x) >0 \ \ {\rm for} \ \ x > x_0.
\end{equation*}
Moreover ${U^{\e}_{{}_{M}}}'(x) \leq 1$ for $x \in (0,\ell)$ and
\begin{equation*}
 {U^{\e}_{{}_{M}}}''(x) >0 \ \ {\rm for} \ \ x<x_0  \ \ {\rm and } \ \ {U^{\e}_{{}_{M}}}''(x) <0 \ \ {\rm for}\ \ x>x_0.
\end{equation*}
\end{proposition}

\begin{remark} { \rm
Because of the assumption \eqref{ipof}$_3$ on the symmetry of the flux function $f$, it turns out that $x_0 \equiv \ell/2$. }
\end{remark}

\begin{proof}

 Let $x_0 \in (0,\ell)$, and let us consider the interval $(x_0, \ell)$. The proof of the statement follows by using the same arguments as in the proof of Proposition \ref{propstazpos}, and by choosing  $v_1(x)=x-x_0$ and  $v_2(x)=\alpha \sin \left( \frac{\pi}{\ell-x_0} (x-x_0)\right)$ as supersolution and subsolution respectively.

In particular, there exists a positive solution $u_+(x)$ to \eqref{ForcBurgStat} in $(x_0,\ell)$ such that $u_+(x_0)=0$. A symmetric argument can be used inside the interval $(0,x_0)$ to prove the existence of a negative solution $u_-(x)$, so that
\begin{equation}\label{defU}
{U^{\e}_{{}_{M}}}(x)=\left\{ \begin{aligned}
&u_- (x)\ \ {\rm for} \ \ x<x_0 \\
&u_+(x) \ \ {\rm for} \ \ x>x_0 \\ 
\end{aligned}\right.
\end{equation}
and $U^\varepsilon_{{}_{M}}(x_0)=0$. Because of the assumption $\eqref{ipof}_{3}$ on the symmetry of the flux function $f$, then $x_0 \equiv \ell/2$. In particular, the steady state $U^\e_{{}_{M}}$ is a $C^1$-matched function. The unicity follows immediately from the unicity of $U_{\e,+}$.

For the proof of the second part of the Proposition, it is enough to give a description of the positive solution $u_+(x)$ for $x \in (x_0,\ell)$. The same arguments can be used for the symmetric case of $u_-(x)$ in the interval $(0,x_0)$.

Again, the proof is identically to the one of Proposition \ref{propstazpos}, by considering the interval $(x_0,\ell)$ instead of $(0,\ell)$.

\end{proof}

\begin{remark}{\rm
The end of the proof of Proposition \ref{prop:exstat} is justified by the symmetry properties of the solutions to \eqref{ForcBurgStat}; indeed, if we consider the interval $(a,b) \subset (0,\ell)$, and if we define $\bar U^+(x;a,b)$  as the unique positive solution to
\begin{equation*}
\left\{\begin{aligned}
		\varepsilon\,\partial_x^2u &= \partial_x f(u)- f'(u)
		\\
 		u(a)	& =u(b)=0, \quad x\in (a,b),	\\
 	 \end{aligned}\right.
\end{equation*}
the solutions to \eqref{ForcBurgStat} can be defined starting from $\bar U^+(x;a,b)$. For example
\begin{equation*}
 U^\e_{{}_{M}}(x):=\left\{\begin{aligned}
-\bar U^+&(\ell/2-x;0,\ell/2)\ \ &{\rm for} \ \ &x\in(0,\ell/2)  \\ 
&\bar U^+(x;\ell/2,\ell) \ \ &{\rm for} \ \ &x \in (\ell/2,\ell).
 \end{aligned}\right.
\end{equation*}

}
\end{remark}

\vskip0.5cm

The following results characterizes the behavior of $U_{\e,+}$ and $U^\e_{{}_{M}}$ with respect to the parameter $\e$.
\begin{proposition}
The solution $U_{\e,+}(x)$ converges pointwise to the function  $x$ in $(0,\ell)$ when  $\e \to 0$.
\end{proposition}
\begin{proposition}\label{convergenza}
The solution $U^\e_{{}_{M}}(x)$ converges pointwise to the function $x-\ell/2$ in $(0,\ell)$ when $\e \to 0$.

\end{proposition}

We only prove the convergence property of the metastable steady state $U^\e_{{}_{M}}$, the one we are interested the most. Again, a minor modification of the argument can be used in the other cases. In order to prove Proposition \ref{convergenza}, we need to state and prove the following Lemma.
\begin{lemma}\label{introconv}
Let $\e < \e_1$ and let $U_{\e,+}$ and $U_{\e_1,+ }$ the (unique) positive solutions to \eqref{ForcBurgStat} with $\e$ and $\e_1$ respectively. Then $U_{\e_1,+}(x) < U_{\e,+}(x)$ for all $x \in (0,\ell)$.
\end{lemma}
\begin{proof}
Let $x \in (0,\ell)$. Since $U^{''}_{\e_1,+}(x) <0$ we note that, if $\e < \e_1$, then  $U_{\e_1,+}$ is a subsolution for \eqref{ForcBurgStat}; moreover, the exists a  larger supersolution that is $h(x)=x$. Hence, by uniqueness, there follows $U_{\e_1,+} <U_{\e,+}$.
\end{proof}

\begin{proof}[{\bf Proof of Proposition \ref{convergenza}}]
Because of Lemma \ref{introconv} and because of the symmetry properties of the problem,  if $\e < \e_1$, then $U^{\e}_{{}_{M}} > U^{\e_1}_{{}_{M}}$ for $x > \ell/2$, while
$U^{\e}_{{}_{M}} < U^{\e_1}_{{}_{M}}$ for $x < \ell/2$. Hence, since the function $v=x-\ell/2$ is a subsolution (supersolution respectively ) in the interval $(0,\ell/2)$ (in the interval $(\ell/2,\ell)$ respectively), we have 
\begin{equation*}
\begin{aligned}
&\lim_{\e \to 0} U^\e_{{}_{M}}(x) = L(x) \geq x-\ell/2 \ \ {\rm for} \ \ x \in (0,\ell/2), \\
&\lim_{\e \to 0} U^\e_{{}_{M}}(x) = l(x) \leq x-\ell/2 \ \ {\rm for} \ \ x \in (\ell/2,\ell).
\end{aligned}
\end{equation*}
Now we show that $l(x)=L(x)=x-\ell/2$. To this aim, let $a \in (0,\ell/2)$, $ c \in (\ell/2,\ell)$, $b= (\ell/2+a)/2$, $d =(\ell+c)/2$. Given $\l \in (0,1)$, let us consider the following function
\begin{equation*}
h^{\l}_{a,c}(x)=\left\{\begin{aligned}
&\frac{\l \,  \, (b-\ell/2)}{a} \,x &{\rm for} \ \  &x \in (0,a) \\
& \omega(x) &{\rm for} \ \  &x \in (a,b) \\
&\lambda (x-\ell/2) &{\rm for} \ \ &x \in (b,c) \\
& \theta(x) &{\rm for} \ \  &x \in (c,d) \\
& \frac{\l (c-\ell/2)}{\ell-d} (\ell-x) &{\rm for} \ \  &x \in (d,\ell)
\end{aligned}\right.
\end{equation*}
where $\omega(x) <0$ and $\theta(x)>0$ are $C^2$ functions such that $h^{\l}_{a,c}$ is a continuous function with continuous derivative in $x=a,b,c,d$. More precisely we require
\begin{equation*}
\begin{aligned}
&\omega(a)=\omega(b), \quad \omega'(a)= \frac{\l (b-\ell/2)}{a}, \ \ \omega'(b)=\l, \quad \omega''(a)=\omega''(b)=0, \\
&\theta(c)=\theta(d), \quad \theta'(c)= \l,  \ \ \theta'(d)=\frac{\l (\ell/2-c)}{\ell-d}, \quad \theta''(c)=\theta''(d)=0.
\end{aligned}
\end{equation*}
Under these hypotheses, it is easy to check that $h^{\l}_{a,c}$ is a supersolution to \eqref{ForcBurgStat} for $x \in (0,\ell/2)$, and a  subsolution for $x \in (\ell/2,\ell)$. Hence, for $\e$ small enough, we deduce
\begin{equation*}
\begin{aligned}
x-\ell/2 &\leq U^\e_{{}_{M}}(x) \leq h^{\l}_{a,c}(x) &{\rm for} \ \  &x \in (0,\ell/2), \\
h^{\l}_{a,c}(x) &\leq U^\e_{{}_{M}}(x) \leq x-\ell/2 &{\rm for} \ \  &x \in (\ell/2,\ell).
\end{aligned}
\end{equation*}
Since $\l$ can be chosen arbitrarily close to $1$, while $a$ and $c$ can be chosen arbitrarily close to $0$ and $\ell$ respectively, it follows that $U^\e_{{}_{M}}(x)$ converges pointwise to $x-\ell/2$ as $\e \to 0$ for all $x \in (0,\ell)$.

\end{proof}

Starting from the properties of $U_{\e,+}$ it is possible to prove similar results for the other solutions to \eqref{ForcBurgStat}, by making use of symmetries and scalings in the problem. Furthermore, equilibrium solutions with more than one zero crossing are also possible (see \cite{Goo94} for more details).

\begin{proposition}\label{Propaltre}
Concerning the solutions to \eqref{ForcBurgStat}, there exists a unique negative solution $U_{\e,-}(x):= -U_{\e,+}(\ell-x)$ such that
\begin{itemize}
\item $U'_{\e,-} \leq 1$, $x-\ell< U_{\e,-} \leq 0$ and $U''_{\e,-} \geq 0$ for $x \in (0,\ell)$.
\vskip0.1cm
\item $U_{\e,-}$ converges to $x-\ell$ in $(0,\ell)$ pointwise when  $\e \to 0$.
\vskip0.1cm
\end{itemize}
Additionally, there exists a unique stationary solution $U^\e_{{}_{NS}}$ that has one zero inside the interval $(0,\ell)$ and such that
\begin{itemize}
\item $0<{U^{\e}_{{}_{NS}}}(x) <x$ for $x < \ell/2$, and $x-\ell<{U^{\e}_{{}_{NS}}}(x) <0$ for $x > \ell/2$
\vskip0.1cm
\item ${U^{\e}_{{}_{NS}}}$ converges to the function 
\begin{equation*}
U^0_{{}_{NS}}(x):= \left\{ \begin{aligned}
& x \quad & x \in (0,\ell/2) \\
& x-\ell \quad & x \in (\ell/2,\ell)
\end{aligned}\right.
\end{equation*}
pointwise for $x\neq \ell/2$ when $\varepsilon \to 0$.
\vskip0.1cm
\end{itemize}

\end{proposition}

\vskip0.5cm

\subsection{Stability of stationary solutions}

\begin{ans}
A stationary solution $v$ to \eqref{ForcBurg} is stable if for any $\e >0$ there exists $\delta =\delta(\e)>0$ such that, if $|u_0(x)-v(x)|_{{}_{L^\infty}} < \delta$, then there exists a time $T \geq 0$ such that $|u(x,t)-v(x)|_{{}_{L^\infty}} < \e$ for all $t\geq T$. Finally, $v$ is unstable if it is not stable.
\end{ans}

Concerning the stability properties of the stationary solutions to \eqref{ForcBurg}, the following Proposition holds.
\begin{proposition}\label{prop:stab}
The stationary solutions $U_{\e,\pm}(x)$ are stable, while $U^\e_{{}_{M}}(x)$ and $U^\e_{{}_{NS}}(x)$ are unstable.
\end{proposition}

The proof of the Proposition \ref{prop:stab} is based on the following well known result concerning the evolution problem \eqref{ForcBurg} with a sub or supersolution as initial datum (see \cite{AroCraPel82,Sat72} and \cite[Theorem 6.4]{BerKamSiv01}).
\begin{proposition}
Let $v(x)$ be a weak subsolution (respectively supersolution) to \eqref{ForcBurgStat}. Let $u(x,t)$ be the solution to \eqref{ForcBurg} with initial datum $u_0(x)=v(x)$. Then, for $t \to +\infty$, $u(x,t)$ converges monotonically to a stationary solution $U(x)$ to \eqref{ForcBurg}, i.e. $u(x,t) \nearrow U(x)$ ( or $u(x,t) \searrow U(x)$). 
\end{proposition}

\begin{proof}[{\bf Proof of Proposition \ref{prop:stab}}]
We prove only the instability property of $U^\e_{{}_{M}}$. To prove the results stated for $U^\e_{{}_{NS}}$ and $U_{\e,\pm}$, the basic idea is the same, and follows the lines of  \cite[Theorem 6.4]{BerKamSiv01}.
\vskip0.1cm

Given $(a,b) \subset (0,\ell)$, let $\bar U^\pm(x;a,b)$ the unique positive (respectively negative) solution to
\begin{equation*}
\left\{\begin{aligned}
		\varepsilon\,\partial_x^2u &= \partial_x f(u)+ f'(u)
		&\qquad &x\in (a,b)\\
 		u(a)	& =u(b)=0	  \\
 	 \end{aligned}\right.
\end{equation*}
Moreover, given $0<\alpha<\beta<\g<\ell$, let us define
\begin{equation*}
v(x)=\left\{ \begin{aligned}
&\bar U^-(x;\alpha,\beta) \ \ &{\rm for} \ \ &x \in [\alpha,\beta] \\ 
&\bar U^+(x;\beta,\gamma)\ \ &{\rm for} \ \ &x\in [\beta,\g] \\
& 0 \ \ &{\rm for} \ \ &x\in [\g,\ell]
\end{aligned}\right.
\end{equation*}
Hence, if $\alpha=0$, $0<\beta<\ell/2$ and $\g=\ell$, we know that $U^\e_{{}_{M}}(x) \leq v(x)$. Now let $v(x,t)$  be the solution to \eqref{ForcBurg} with initial datum $v_0(x)=v(x)$. We have $v(x,t) \nearrow U_{\e,+}$ as $t \to +\infty$, since there are no other stationary solutions $U(x)$ such that $v \leq U$. This prove the instability of $U^\e_{{}_{M}}$, since $v(x) \to U^\e_{{}_{M}}$ as $\beta \to \ell/2$.

\end{proof}

As a consequence of the instability of $U^\e_{{}_{M}}(x)$, if we start from an initial datum $u_0(x)$ close to such unstable configuration, we will see in finite time that the solution $u(x,t)$ to \eqref{ForcBurg} ``run away" from $U^\e_{{}_{M}}$. Precisely, there exists $\delta>0$ and a time $T >0$ such that
\begin{equation*}
|u(x,t)-U^\varepsilon_{{}_{M}}(x)| > \delta, \quad \forall \ \ x \in I, \ \ \forall \ \ t \geq T.
\end{equation*}
However, as already stressed before , solutions to \eqref{ForcBurg} generated by initial data close to $U^\e_{{}_{M}}$ exhibit a metastable behavior, i.e. the convergence to one of the equilibrium configuration $U_{\e,\pm}$ is exponentially slow in time. To rigorously describe such behavior, from now on we will consider equation \eqref{ForcBurg} together with continuous  initial data of the form
\begin{equation*}
u_0(x) <0 \ \ {\rm in} \ \ (0,a_0), \quad u_0(x) >0 \ \ {\rm in} \ \ (a_0, \ell), \ \ \ {\rm for \ some } \ \ a_0 \in (0,\ell).
\end{equation*}

\section{The metastable dynamics and the linearized problem}\label{sect:general}

In this Section we analyze the solution $u$ to \eqref{ForcBurg} in the vanishing viscosity limit, i.e. $\varepsilon\to0$. In particular, we address the question whether a phenomenon of metastability occurs for such kind of problem, and how this special dynamics is related to the viscosity coefficient and to the initial datum $u_0$.

Let us define the nonlinear differential operator 
$$\mathcal P^\e[u]:=\varepsilon\,\partial_x^2u - \partial_x f(u)+ f'(u),$$ 
that depends singularly on the parameter $\e$, meaning that $\mathcal P^0[u]$ is of lower order.

Our primarily assumption is the following: we suppose that there exists a one-parameter family of functions $\{ U^\varepsilon(x;\xi) \}_{\xi \in I}$ such that
\begin{itemize}
\item The nonlinear term $\mathcal P^\e[U^\e]$ is small in $\e$ in a sense that we will specify later.
\item Each element of the family is  such that
\begin{equation*}
\left\{\begin{aligned}
&U^\e(x;\xi) <0 \quad {\rm for} \quad  0< x< \xi, \\
& U^\e(x;\xi) >0 \quad {\rm for} \quad  \xi<x<\ell.
\end{aligned}\right.
\end{equation*}
 \item There exists a value $\bar \xi \in I$ such that the element of the family $U^\varepsilon(x;\bar \xi)$ corresponds to a stable steady state to \eqref{ForcBurg}.
\end{itemize}
The parameter $\xi$ describes the unique zero of $U^\e$, corresponding to the location of the interface; under these hypotheses, starting from an initial configuration close to $U^\e$, a metastable behavior for the time-dependent solution is expected.

\vskip0.2cm
The family $\{ U^\varepsilon(x;\xi) \}_{\xi \in I}$ can be seen as a {\bf family of approximate steady states} for \eqref{ForcBurg}, in the sense that each element satisfies the stationary
equation up to an error that is small in $\e$. More precisely,  we ask for the existence of a family of smooth positive functions $\Omega^\e(\xi)$, that converge to zero as $\e \to 0$, uniformly with respect to $\xi$, and such that 
\begin{equation*}
|\langle \psi(\cdot ), \mathcal P^\e[U^\e(\cdot;\xi)]\rangle| \leq \Omega^\e(\xi)|\psi|_{L^\infty}, \quad \forall \, \psi \in C(I), \, \forall \, \xi \in I.
\end{equation*}
We  also consider the following additional assumption specifying the structure of the term $\Omega^\varepsilon$: we require that there exists a family of smooth positive 
functions $\omega^\varepsilon=\omega^\varepsilon(\xi)$, uniformly convergent to zero as 
$\varepsilon\to 0$,  such that
\begin{equation*}
	\Omega^\varepsilon(\xi)\leq \omega^\varepsilon(\xi)\,|\xi-\bar \xi|.
\end{equation*}
This hypothesis incorporates the fact that,
for a distinct value  $\bar \xi\in I$,  the element $U^\varepsilon(\cdot,\bar\xi)$ solves the stationary equation. The dependence of $\Omega^\varepsilon$ and $\omega^\varepsilon$ on 
$\varepsilon$ is here crucial, since they measure how far is an element of the family $U^\e$ from being an exact stationary solution.

Let us stress that, differently to the construction in  \cite{SunWard99}, 
where the approximate stationary solutions satisfy exactly the equation and the boundary condition to 
within exponentially small terms, here we assume that the generic element $U^{\varepsilon}$ satisfies 
the boundary conditions exactly and the equation approximately.

\vskip0.2cm

Once the one-parameter family $\{ U^\varepsilon(x;\xi) \}_{\xi \in I}$ is chosen, we look for a solution to \eqref{ForcBurg} in the form
\begin{equation}\label{formau}
u(x,t)= U^\e(x;\xi(t))+v(x,t),
\end{equation}
where  the perturbation $v \in L^2(I)$ is determined by the difference between the solution $u$ and an element of the family of approximate steady states.

The idea of a linearization around $U^\e$ is developed in order to separate the two distinct phases of the dynamics of the solution. Firstly, we mean to understand what happens far from the stable equilibrium solution when the interface is formed; subsequently, we want to follow its evolution towards the asymptotic limit. To this aim, we suppose that the parameter $\xi$, describing the unique zero of the ``quasi-stationary" solution $U^\e$, depends on time, so that its evolution towards one of the wall $x=0$ or $x=\ell$ (corresponding to the equilibrium solutions $U_{\e,+}$ and $U_{\e,-}$ respectively) describes the asymptotic convergence of the interface towards the equilibrium. Hence, our purpose is to determine an equation for the value $\xi(t)$, characterizing the metastable behavior. 

By substituting \eqref{formau} into \eqref{ForcBurg}, we obtain
\begin{equation}\label{eqvNL}
\d_t v = \mathcal L^\e_{\xi(t)} v + \mathcal 	P^\e[U^\e(\cdot;\xi)] -\d_\xi U^\e(\cdot;\xi) \, \frac{d\xi}{dt}+ \mathcal Q^\e[v,\xi],
\end{equation}
where
\begin{equation*}
\begin{aligned}
\mathcal L^\e_{\xi(t)} v  := \e \d_x^2 v- \d_x(f'(U^\e) v)+ f''(U^\e) v
\end{aligned}
\end{equation*}
is the linearized operator arising from the linearization around $U^\e$, while $\mathcal Q^\e[v,\xi]$ collects  quadratic terms in $v$, and, since we are assuming $v$ to be small, it is obtained by disregarding higher order terms in the variable. Precisely,
\begin{equation*}
\mathcal Q^\e[v,\xi]:= \frac{1}{2} \left\{  -\d_x \left( f''(U^\e) v^2  \right)+ f'''(U^\e)v^2\right\}.
\end{equation*}

\subsection{Spectral hypotheses and the projection method}

We begin by analyzing the spectrum of the linearized operator $\mathcal L^\e_\xi$. The eigenvalue problem reads
\begin{equation*}
\e \d_x^2 \varphi- \d_x(f'(U^\e) \varphi)+ f''(U^\e) \varphi= \lambda^\e \varphi, \quad \varphi(0)=\varphi( \ell)=0.
\end{equation*}
Firstly, we show that the eigenvalues of $\mathcal L^\e_\xi$ are real. To this aim, let us introduce the self-adjoint operator
\begin{equation*}
\mathcal M^\e_{\xi(t)} \psi := \e^2 \d_x^2 \psi - a^\e(x;\xi(t)) \psi+\e f''(U^\e)\psi, \quad a^\e(x;\xi(t)):= \left(\frac{f'(U^\e)}{2}\right)^2+\frac{1}{2} \, \e \, \d_x f'(U^\e).
\end{equation*}
It is easy to check that $\varphi^\e$ is an eigenfunction for $\mathcal L^\e_\xi$ relative to the eigenvalue $\lambda^\e$ if and only if
\begin{equation*}
\psi^\e(x;\xi)= \exp\left( -\frac{1}{2\e} \int_{x_0}^x f'(U^\e)(t;\xi) dt \right) \varphi^\e(x;\xi)
\end{equation*}
is an eigenfunction for the operator $\mathcal M^\e_\xi$ relative to the eigenvalue $\mu^\e = \e \lambda^\e$. Hence
\begin{equation*}
\e \, \sigma(\mathcal L^\e_\xi ) \equiv \sigma(\mathcal M^\e_\xi ),
\end{equation*}
so that, since  $\mathcal M^\e_\xi$ is self-adjoint, we can state the spectrum of $\mathcal L^\e_\xi$ is composed by {\bf real eigenvalues}.

Moreover, we assume the spectrum of $\mathcal L^\e_\xi$ to be composed of a decreasing sequence $\{ \l^\e_k(\xi)\}_{k \in \N}$ of real eigenvalues such that
\begin{itemize}
\item  $\l^\e_1(\xi) \to 0$ as $\e \to 0$, uniformly with respect to $\xi$.
\item All the eigenvalues $\{ \l^\e_k \}_{k\geq 2}$ are negative and there exist constants $C,C'$ such that
\begin{equation*}
\lambda^\e_1(\xi)-\lambda_2^\e(\xi) \geq C' \ \ \ \ \ \forall \ \xi \in I, \quad \l_k^\e(\xi) \leq -C k^2, \ \ \ {\rm for} \ \ k \geq 2.
\end{equation*}
\end{itemize}
Hence, we assume that there is a spectral gap between the first and the second eigenvalue. Moreover, we ask for $\l_1^\e$ to be small in $\e$ (uniformly with respect to $\xi$) and we assume the sequence $\{\l_k^\e\}_{k \geq 2}$ to diverge to $-\infty$ as $-k^2$. 

\begin{remark}{\rm
We note that there are no requests on the sign of the first eigenvalue $\l^\e_1$;  in section 2 we have proven the instability of $U^\e_{{}_{M}}$, so that, since $U^\e$ well approximates the exact steady state $U^\e_{{}_{M}}$, we can state that {\bf the first eigenvalue $\l^\e_1$ is  positive}. The metastable behavior is indeed a consequence of the smallness, with respect to $\e$, of such first eigenvalue.}
\end{remark}

In order to obtain a differential equation for the parameter $\xi$, we use an adapted version of the {\rm projection method}: since we have supposed the first eigenvalue of the linearized operator to be small in $\e$, i.e. $\lambda_1^\e \to 0$ as $\e \to 0$, a necessary condition for the solvability of \eqref{eqvNL} is that  the first component of the solution has to be zero. More precisely, in order to remove the singular part of the operator $\mathcal L^\e_\xi$ in the limit $\e \to 0$, we set an algebraic condition ensuring orthogonality between $\psi^\e_1$ and $v$, so that the equation for the parameter $\xi(t)$ is chosen in such a way that the unique growing terms in the perturbation $v$ are canceled out.  Setting $v_k=v_k(\xi;t):=\langle \psi^\varepsilon_k(\cdot;\xi),v(\cdot,t)\rangle$, we thus impose
\begin{equation*}
\frac{d}{dt} \langle \psi^\varepsilon_1(\cdot;\xi(t)), v(\cdot,t) \rangle =0
	\qquad\textrm{and}\qquad
	\langle \psi^\varepsilon_1(\cdot;\xi_0), v_0(\cdot)\rangle=0.
\end{equation*}
Using equation \eqref{eqvNL}, we have
\begin{equation*}
	\langle \psi^\varepsilon_1(\xi,\cdot),\mathcal{L}^\varepsilon_\xi v+ {\mathcal P}^\e[U^{\varepsilon}(\cdot;\xi)]
		- \partial_{\xi}U^{\varepsilon}(\cdot;\xi)\frac{d\xi}{dt}+{\mathcal Q}^\varepsilon[v,\xi] \rangle 
		+ \langle  \partial_{\xi}\psi^\varepsilon_1(\xi,\cdot) \frac{d\xi}{dt},v \rangle =0.
\end{equation*}
Since $\langle \psi^\varepsilon_1, {\mathcal L}^\e_{\xi}v \rangle= \lambda^\e_1\langle \psi^\varepsilon_1, v \rangle=0$, 
we obtain a scalar nonlinear differential equation for the variable $\xi$,  that is
\begin{equation}\label{eqxi0}
		\frac{d\xi}{dt}=\frac{\langle \psi^\varepsilon_1(\cdot;\xi),
			{\mathcal P}^\e[U^{\varepsilon}(\cdot;\xi)]+\mathcal{Q}^\varepsilon[v,\xi] \rangle}{ \langle \psi^\varepsilon_1(\cdot;\xi), \partial_{\xi}U^{\varepsilon}(\cdot;\xi) \rangle - \langle  \partial_{\xi}\psi^\varepsilon_1(\cdot;\xi),v \rangle }.
\end{equation}
We notice that if $U^\varepsilon(\cdot;\bar\xi)$ is the exact stationary solution, then
\begin{equation*}
	\mathcal{P}[U^\varepsilon(\cdot;\xi)]
	=\mathcal{P}[U^\varepsilon(\cdot;\xi)]-\mathcal{P}[U^\varepsilon(\cdot;\bar\xi)]
	\approx\mathcal{L}_\xi^\varepsilon\partial_\xi U^\varepsilon(\cdot;\bar\xi)(\xi-\bar\xi).
\end{equation*}
The fact that $\mathcal L^\e_\xi ( \d_\xi U^\e)$ is uniformly small suggests that the first eigenfunction $\psi^\e_1$ is proportional to $\partial_\xi U^\varepsilon$ (at least for small $\e$), so that we can renormalize the first adjoint eigenfunction in such a way
\begin{equation*}
\langle \psi^\varepsilon_1(\cdot;\xi), \partial_{\xi}U^{\varepsilon}(\cdot;\xi) \rangle  =1, \quad \forall \ \xi \in I.
\end{equation*}
Since we consider  a small perturbation, in the regime $v \to 0$ we have
\begin{equation*}
	\frac{1}{1-\langle  \partial_{\xi}\psi^\varepsilon_1(\cdot;\xi),v \rangle} 
		= 1+\langle \partial_{\xi} \psi^\varepsilon_1,v \rangle+ R[v],
\end{equation*}
where the remainder $R$ is of order $o(|v|)$, and it is defined as
\begin{equation*}
R[v]:= \frac{\langle  \partial_{\xi}\psi^\varepsilon_1(\cdot;\xi),v \rangle^2}{1-\langle  \partial_{\xi}\psi^\varepsilon_1(\cdot;\xi),v \rangle}.
\end{equation*}
Inserting in \eqref{eqxi0}, we end up with the following nonlinear equation for $\xi$
\begin{equation}\label{eqxiNL}
	\frac{d\xi}{dt}=\theta^\varepsilon(\xi)\bigl(1+\langle\partial_{\xi} \psi^\varepsilon_1, v \rangle\bigr)
		+ \rho^\varepsilon[\xi,v], 
\end{equation}
where
\begin{equation*}
	\begin{aligned}
 	\theta^\varepsilon(\xi)
		&:=\langle \psi^\varepsilon_1,{\mathcal P^\e[U^{\varepsilon}] \rangle},\\
	\rho^\varepsilon[\xi,v]&:=  \langle \psi^\e_1, \mathcal Q^\e[v,\xi] \rangle \bigl(1+\langle\partial_{\xi} \psi^\varepsilon_1, v \rangle\bigr)+ \langle \psi^\e_1, \mathcal P^\e[U^\e]+ \mathcal Q^\e[v,\xi] \rangle R[v].
	\end{aligned}
\end{equation*}
Moreover, plugging \eqref{eqxiNL} into \eqref{eqvNL}, we obtain a partial differential equation for the perturbation $v$
\begin{equation}\label{eqvNLfin}
\partial_t v= H^\varepsilon(x;\xi)
		+ ({\mathcal L}^\varepsilon_\xi+{\mathcal M}^\varepsilon_\xi)v
			+\mathcal{R}^\varepsilon[v,\xi],
\end{equation}
where 
\begin{align*}
		H^\varepsilon(\cdot;\xi)&:={\mathcal P}^\varepsilon[U^{\varepsilon}(\cdot;\xi)]
			-\partial_{\xi}U^{\varepsilon}(\cdot;\xi)\,\theta^\varepsilon(\xi),\\
		{\mathcal M}^\varepsilon_\xi v&:=-\partial_{\xi}U^{\varepsilon}(\cdot;\xi)
			\,\theta^\varepsilon(\xi)\,\langle\partial_{\xi} \psi^\varepsilon_1, v \rangle,\\ 
		\mathcal{R}^\varepsilon[v,\xi]&:=\mathcal{Q}^\varepsilon[v,\xi]
								-\partial_{\xi}U^{\varepsilon}(\cdot;\xi)\,\rho^\varepsilon[\xi,v].
\end{align*}

\section{The slow motion of the interface location}

\subsection{Analysis of the linearized system}\label{Secmeta:line}
Equations \eqref{eqxiNL}-\eqref{eqvNLfin}  form a coupled system for the couple $(\xi,v)$. This system is obtained by linearizing with respect to $v$, and by keeping the nonlinear dependence on $\xi$, in order to describe the evolution of the interface when it is localized far from the equilibrium location. Hence the terms arising from the linearization around $v \sim 0$ are asymptotically smaller than the other terms and can be neglected, so that in the following we will consider the following reduced system
\begin{equation}\label{qulinsys}
\left\{\begin{aligned}
\frac{d\xi}{dt}&=\theta^\varepsilon(\xi)\bigl(1+\langle\partial_{\xi} \psi^\varepsilon_1, v \rangle\bigr), \\
\partial_t v&= H^\varepsilon(x;\xi)
		+ ({\mathcal L}^\varepsilon_\xi+{\mathcal M}^\varepsilon_\xi)v,
\end{aligned}\right.
\end{equation}
where the $o(v)$ order terms have been canceled out, together with initial data
\begin{equation}\label{Idata}
\langle \psi^\e_1(\cdot;\ \xi_0), v_0 \rangle =0, \quad v_0 = u_0- U^\e(\cdot;\xi_0).
\end{equation}
We mean to analyze the behavior of the solution to \eqref{qulinsys} in the limit of small $\e$. In order to state our first result,
let us recall the hypotheses we assumed on the terms of such system.

\vskip0.2cm
{\bf H1.}There exists a family  of approximate steady states $\{ U^\e(x;\xi) \}_{\xi \in I}$ such that

\begin{itemize}

\item There exists a value $\bar \xi \in I$ such that the element $U^\e(x;\bar \xi)$ corresponds to a stable steady state for the original equation.

\item For every $\xi \in I$, each element of the family satisfies
\begin{equation*}
U^\e(x;\xi) >0 \ \ \ {\rm for} \ \ \ x \in (0,\xi) \quad {\rm and} \quad  U^\e(x;\xi) <0 \ \ \ {\rm for} \ \ \ x \in (\xi,\ell).
\end{equation*}

\item There exists a family of smooth and positive functions $\Omega^\e(\xi)$, uniformly converging  to zero as $\e \to 0$, such that there holds
\begin{equation*}
|\langle \psi(\cdot ), \mathcal P^\e[U^\e(\cdot;\xi)]\rangle| \leq \Omega^\e(\xi)|\psi|_{L^\infty}, \quad \forall \, \psi \in C(I), \, \forall \, \xi \in I.
\end{equation*}

\item There exists a family of smooth positive 
functions $\omega^\varepsilon=\omega^\varepsilon(\xi)$, uniformly convergent to zero as 
$\varepsilon\to 0$,  such that
\begin{equation*}
	\Omega^\varepsilon(\xi)\leq \omega^\varepsilon(\xi)\,|\xi-\bar \xi|.
\end{equation*}

\end{itemize}

\vskip0.2cm
{\bf H2.}  The sequence of eigenvalues $\{ \lambda^\e_k(\xi)\}_{k \in \N}$ of the linearized operator $\mathcal L^\e_\xi$ is such  that
\begin{itemize}
\item  $\l^\e_1(\xi) \to 0$ as $\e \to 0$ uniformly with respect to $\xi$.
\vskip0.1cm
\item All the eigenvalues $\{ \l^\e_k \}_{k\geq 2}$ are negative and there exist constants $C,C'$ such that
\begin{equation*}
\lambda^\e_1(\xi)-\lambda_2^\e(\xi) \geq C' \ \ \  \forall \ \xi \in I, \qquad \l_k^\e(\xi) \leq -C k^2, \ \ \ {\rm if} \ \ k \geq 2.
\end{equation*}

\end{itemize}

Finally, The eigenfunctions $\varphi^\varepsilon_k(\cdot;\xi)$ and $\psi^\varepsilon_k(\cdot;\xi)$
are normalized so that
\begin{equation*}
	\langle \psi^\varepsilon_1(\cdot;\xi), \partial_{\xi}U^{\varepsilon}(\cdot;\xi)\rangle=1
	\qquad\textrm{and}\qquad
	\langle \psi^\varepsilon_j, \varphi^\varepsilon_k\rangle
	=\left\{\begin{aligned} & 1	&\qquad &\textrm{if }j=k,\\ & 0	&\qquad &\textrm{if }j\neq k,
		\end{aligned}\right.
\end{equation*}

and we assume 
\begin{equation}\label{derpsiphi}
	\sum_{j} \langle \partial_\xi \psi^\varepsilon_k, \varphi^\varepsilon_j\rangle^2
	=\sum_{j} \langle \psi^\varepsilon_k, \partial_\xi \varphi^\varepsilon_j\rangle^2
	\leq C	\qquad\qquad\forall\,k.
\end{equation}

for some constant $C$ that does not depend on $\xi$.

\vskip0.2cm
 We recall that we denoted by  $\varphi^\varepsilon_k=\varphi^\varepsilon_k(\cdot;\xi)$ the  right eigenfunctions of $\mathcal L^\e_\xi$ and by $\psi^\varepsilon_k=\psi^\varepsilon_k(\cdot;\xi)$ the eigenfunctions of the corresponding  adjoint operator  $\mathcal{L}^{\varepsilon,\ast}_\xi$. In the following, we shall use the notation $\Lambda^\varepsilon_k:=\sup\limits_{\xi\in I} \lambda^\varepsilon_k(\xi)$.

\begin{theorem}\label{thm:metaL}
Let hypotheses {\bf H1-2} be satisfied. Then, denoted by $(\xi,v)$ the solution to the initial-value problem \eqref{qulinsys},
for any $\varepsilon$ sufficiently small, there exists a time $T^\varepsilon$ 
such that for any $t\leq T^\varepsilon$ the solution $v$ can be represented as
\begin{equation*}
 v=z+R,
\end{equation*}
where $z$ is defined by
\begin{equation*}
	z(x,t):=\sum_{k\geq 2} v_k(0)\exp\left(\int_0^t \lambda^\varepsilon_k(\xi(\sigma))\,d\sigma\right)
		\,\varphi^\varepsilon_k(x;\xi(t)),
\end{equation*}
and the remainder $R$ satisfies the estimate
\begin{equation}\label{boundw}
	|R|_{{}_{L^2}}\,\leq C\,|\Omega^\varepsilon|_{{}_{L^\infty}}
		\left(|v_0|_{{}_{L^2}}+1\right),
\end{equation}
for some constant $C>0$. 

Moreover, for $v_0$ sufficiently small in $L^2$, the final time $T^\varepsilon$ can be chosen of the order 
 $C |\L_1^\e|^{-1}$, hence diverging to $+\infty$ as $\e \to 0$.
\end{theorem}

\begin{proof}

The idea of the proof is analogous to the proof of \cite[Theorem 2.1]{MS}.  Setting
\begin{equation*}
	v(x,t)=\sum_{j} v_j(t)\,\varphi^\varepsilon_j(x,\xi(t)),
\end{equation*}
we obtain an infinite-dimensional differential system for the coefficients $v_j$
\begin{equation}\label{eqvk_bis}
	\frac{dv_k}{dt}=\lambda^\varepsilon_k(\xi)\,v_k
		+\langle \psi^\varepsilon_k,F\rangle,
\end{equation}
where
\begin{equation*}
	F:=H^\varepsilon+\sum_{j} v_j\,\Bigl\{{\mathcal M}^\varepsilon_\xi\, \varphi^\varepsilon_j
			-\partial_\xi \varphi^\varepsilon_j\,\frac{d\xi}{dt}\Bigr\}
		=H^\varepsilon-\theta^\varepsilon \sum_{j}\Bigl(a_j+\sum_{\ell} b_{j\ell}\,v_\ell\Bigr)v_j.
\end{equation*}
and the coefficients $a_j$, $b_{jk}$ are given by
\begin{equation*}
	a_j:=\langle \partial_{\xi} \psi^\varepsilon_1, \varphi^\varepsilon_j\rangle\,
			\partial_{\xi}U^{\varepsilon}+\partial_\xi \varphi^\varepsilon_j,
	\qquad
	b_{j\ell}:=\langle \partial_{\xi} \psi^\varepsilon_1, \varphi^\varepsilon_\ell\rangle
			\,\partial_\xi \varphi^\varepsilon_j.
\end{equation*}
Convergence of the series is guaranteed by assumption \eqref{derpsiphi}. 

Now let us set
\begin{equation*}
	E_k(s,t):=\exp\left( \int_s^t \lambda_k^\varepsilon(\xi(\sigma))d\sigma\right).
\end{equation*}
Note that, for $0\leq s<t$, there hold
\begin{equation*}
	E_k(s,t)=\frac{E_k(0,t)}{E_k(0,s)}
		\qquad\textrm{and}\qquad	
	0\leq E_k(s,t)\leq e^{\Lambda_k(t-s)}.
\end{equation*}
By differentiating \eqref{derpsiphi}, we deduce
\begin{equation}\label{antisim}
	\langle \partial_\xi \psi^\varepsilon_j, \varphi^\varepsilon_k\rangle
	+\langle \psi^\varepsilon_j,  \partial_\xi \varphi^\varepsilon_k\rangle=0.
\end{equation}
Hence, for the coefficients $a_j$  
there holds
\begin{equation*}
	\langle \psi^\varepsilon_k, a_j\rangle 
		=\langle \partial_{\xi} \psi^\varepsilon_1, \varphi^\varepsilon_j\rangle\,
			\bigl(\langle \psi^\varepsilon_k, \partial_{\xi}U^{\varepsilon}\rangle-1\bigr),
\end{equation*}
so that, since $\langle \psi^\varepsilon_1, a_j\rangle=0$ for any $j$, equation \eqref{eqvk_bis} for $k=1$ simplifies to 
\begin{equation}\label{eqw1}
	\frac{dv_1}{dt}=\lambda^\varepsilon_1(\xi)\,v_1
		-\theta^\varepsilon(\xi) \sum_{\ell, j} \langle \psi^\varepsilon_1, b_{j\ell}\rangle \,v_\ell\,v_j.
\end{equation}
Choosing $v_1(0)=0$, it follows
\begin{equation}\label{reprform}
	\begin{aligned}
	v_1(t)&=-\int_0^t \theta^\varepsilon(\xi)
		\sum_{\ell, j} \langle \psi^\varepsilon_1, b_{j\ell}\rangle \,v_\ell\,v_j,
		\,E_1(s,t)\,ds\\
	v_k(t)&=v_k(0)\,E_k(0,t)\\
		& +\int_0^t \Bigl\{\langle \psi^\varepsilon_k,H^\varepsilon\rangle
			-\theta^\varepsilon(\xi)\sum_{j}\Bigl(\langle \psi^\varepsilon_k, a_j\rangle 
			+\sum_{\ell} \langle \psi^\varepsilon_k, b_{j\ell}\rangle \,v_\ell\Bigr) v_j
			\Bigr\} E_k(s,t)\,ds,
	\end{aligned}
\end{equation}
for $k\geq 2$.
Let us introduce the function 
\begin{equation*}
	z(x,t):=\sum_{k\geq 2} v_k(0)\,E_k(0,t)\,\varphi^\varepsilon_k(x;\xi(t)),
\end{equation*}
which satisfies the estimate $|z|_{{}_{L^2}}\leq |v_0|_{{}_{L^2}} e^{\Lambda^\varepsilon_2\, t}$. Since there hold
\begin{equation*}
 	|\theta^\varepsilon(\xi)|\leq C\,\Omega^\varepsilon(\xi)
		\qquad\textrm{and}\qquad
	|\langle \psi^\varepsilon_k,H^\varepsilon\rangle| \leq C\,\Omega^\varepsilon(\xi)
		\left\{1+ |\langle \psi_k^\varepsilon, \partial_\xi U^\varepsilon\rangle|\right\},
\end{equation*}
after some computations we end up with
\begin{equation*}
	\begin{aligned}
	|v-z|_{{}_{L^2}}(t)&\leq 
		C\int_0^t \Omega^\varepsilon(\xi)|v|_{{}_{L^2}}^2(s)\,E_1(s,t)\,ds
		+C\sum_{k\geq 2}\int_0^t\Omega^\varepsilon(\xi)
			\bigl(1+|v|_{{}_{L^2}}^2(s)\bigr)E_k(s,t)\,ds\\
				&\leq 
		C\int_0^t \Omega^\varepsilon(\xi)\Bigl\{|v|_{{}_{L^2}}^2(s)\,E_1(s,t)
			+\bigl(1+|v|_{{}_{L^2}}^2(s)\bigr)\,\sum_{k\geq 2} E_k(s,t)\Bigr\}\,ds.
	\end{aligned}
\end{equation*}
for some constant $C>0$ depending on the $L^\infty-$norm of $\psi^\varepsilon_k$.
The assumption on the asymptotic behavior of the eigenvalues $\lambda^\e_k$, $k \geq 2$, 
can now be used to bound the series.
Indeed, there holds
\begin{equation*}
	\sum_{k\geq 2} E_k(s,t) \leq E_2(s,t) \sum_{k\geq 2} \frac{E_k(s,t)}{E_2(s,t)}
		\leq C\,(t-s)^{-1/2}\,E_2(s,t) 
\end{equation*}
so that
\begin{align*}
	E_1(t,0) |v-z|_{{}_{L^2}}
		&\leq  C\int_0^t \Omega^\varepsilon(\xi)\Bigl\{|v-z|^2_{{}_{L^2}}(s)\, E_1(s,0)\\
		&\qquad +|z|^2_{{}_{L^2}}(s)\,E_1(s,0)+(t-s)^{-1/2}\,E_2(s,t) E_1(s,0)\Bigr\}\,ds.
\end{align*}
Setting
\begin{equation*}
	N(t):= \sup_{s\in[0,t]} |v-z|_{{}_{L^2}}\,E_1(s,0),
\end{equation*}
then, since $\Lambda^\varepsilon_2\leq \Lambda^\varepsilon_1$,  we obtain 
\begin{equation*}
	\begin{aligned}
	E_1(t,0)|v-z|_{{}_{L^2}} &\leq C\int_0^t \Omega^\varepsilon(\xi)N^2(s)\,E_1(0,s)\,ds \\
		& \quad + C\int_0^t \Omega^\varepsilon(\xi) \Big\{|v_0|^2_{{}_{L^2}}
			 e^{2\Lambda^\varepsilon_2 s}E_1(s,0)  
		+(t-s)^{-1/2}\,E_2(s,t) E_1(s,0) \Bigr\}ds.\\
	\end{aligned}
\end{equation*}
Moreover
\begin{align*}
	&\int_0^t e^{2\Lambda^\varepsilon_2s}E_1(s,0)\,ds
		\leq  E_1(t,0)\int_0^t e^{\Lambda^\varepsilon_2 s}\,ds
		=E_1(t,0)\frac{1}{\Lambda_2^\varepsilon}(e^{\Lambda^\varepsilon_2 s}-1)\leq  \frac{1}{|\Lambda_2^\varepsilon|} E_1(t,0),\\
	&\int_0^t (t-s)^{-1/2}\,E_2(s,t)\,ds
		\leq \int_0^t (t-s)^{-1/2}\,e^{\Lambda_2^\varepsilon\,(t-s)}\,ds
		\leq \frac{1}{|\Lambda_2^\varepsilon|^{1/2}},
\end{align*}
so that, recalling that $\Lambda_2^\varepsilon$ is bounded away from $0$, we deduce
\begin{equation*}
       \begin{aligned}
       &E_1(t,0)|v-z|_{{}_{L^2}} \leq C\Big\{  N^2(t) 
       	\biggl[\int_0^t \Omega^\varepsilon(\xi)\, E_1(0,s) \, ds \biggr]
	+C|\Omega^\varepsilon|_{{}_{\infty}}\left(|v_0|^2_{{}_{L^2}} E_1(t,0)+E_1(t,0)\right) \Big\},
       \end{aligned}
\end{equation*}
that is
\begin{equation*}
	N(t) \leq A N^2(t)+B
		\qquad\textrm{with}\quad
		\left\{\begin{aligned}
		A&:=C\left\{\int_0^t \Omega^\varepsilon(\xi)\, E_1(0,s) \, ds\right\},\\
		B&:=C|\Omega^\varepsilon|_{{}_{L^\infty}} E_1(t,0)\Bigl(|v_0|^2_{{}_{L^2}}+1 \Bigr).
		\end{aligned}\right. 
\end{equation*}
Hence, as soon as
\begin{equation}\label{timecond}
  4AB=4C^2|\Omega^\varepsilon|_{{}_{L^\infty}}E_1(t,0)\Bigl(|v_0|^2_{{}_{L^2}}+1 \Bigr)
	\left(\int_0^t \Omega^\varepsilon(\xi)\, E_1(0,s) \, ds\right)<1,
\end{equation}
there holds
\begin{equation*}
	N(t)\leq  \frac{2B}{1+\sqrt{4AB}}\leq 2B 
		= C\,|\Omega^\varepsilon|_{{}_{L^\infty}}E_1(t,0)\Bigl(|v_0|^2_{{}_{L^2}}+1\Bigr),
\end{equation*}
that means, in term of the difference $v-z$, 
\begin{equation*}
	|v-z|_{{}_{L^2}}\,
		\leq C\,|\Omega^\varepsilon|_{{}_{L^\infty}}
		\Bigl(|v_0|^2_{{}_{L^2}}+1\Bigr).
\end{equation*}
Finally, condition \eqref{timecond} gives a constraint on the final time $T^\varepsilon$. Indeed we ask for
\begin{equation}\label{timecond2}
         4C^2\,|\Omega^\varepsilon|_{{}_{L^\infty}}E_1(t,0)
 	\Bigl(|v_0|^2_{{}_{L^2}}+1 \Bigr) <1
\end{equation}
and
\begin{equation}\label{timecond3}
\int_0^t \Omega^\varepsilon(\xi)\, E_1(0,s) \, ds<1
\end{equation}
to assure condition \eqref{timecond} to be satisfied.
Constraint \eqref{timecond2} can be rewritten as
\begin{equation*}
	 \exp\left(-\int_0^t \lambda^\varepsilon_1(\xi)\,d\sigma\right)
	=E_1(t,0) \leq \frac{C}{|\Omega^\varepsilon|_{{}_{L^\infty}} (|v_0|^2_{{}_{L^2}}+1)},
\end{equation*}
that is, we can choose $T^{\varepsilon,1}$ of the form
\begin{equation*}
 	T^\varepsilon:=\frac{1}{|\Lambda_1^\varepsilon|}
		\ln\left( \frac{C}{|\Omega^\varepsilon|_{{}_{L^\infty}} (|v_0|^2_{{}_{L^2}}+1)}\right)
		\sim C \,|\Lambda_1^\varepsilon|^{-1}
		\ln \left(|\Omega^\varepsilon|^{-1}_{{}_{L^\infty}}\right).
\end{equation*}
On the other hand, from \eqref{timecond3}, we have
\begin{equation*}
\int_0^t \Omega^\varepsilon(\xi)\, E_1(0,s) \, ds \leq |\Omega^\e|_{{}_{L^\infty}} \int_0^t e^{\L_1^\e t} <1,
\end{equation*}
that is
\begin{equation*}
|\Omega^\e|_{{}_{L^\infty}}  |\L_1^\e|^{-1}\left( e^{\L_1^\e t}-1\right) <1 \quad \Rightarrow \quad T^{\e,2} := \frac{1}{|\L_1^\varepsilon|}
		\ln\left(\frac{\L_1^\e}{|\Omega^\varepsilon|_{{}_{L^\infty}}} +1\right) \sim C \, |\L_1^\e|^{-1}.
\end{equation*}
Now the proof is completed.
\end{proof}

Theorem \ref{thm:metaL} gives a very precise estimate for the perturbation $v$. Indeed, we have proven that $v$ has decays properties similar to those of the function $z(x,t)$, hence it converges to zero very fast for $t \to +\infty$; moreover, the difference $|v-z|_{{}_{L^2}}$ is bounded by $|\Omega^\e|_{{}_{L^\infty}}$, meaning that it is small with respect to $\e$.
\begin{remark}{\rm
The previous proof can be easily extend to the case $v \in [L^2(I)]^n$ (see also \cite[Theorem 2.1]{MS}). This is meaningful in light of a possible application of this result in the case of systems.
}
\end{remark}

\vskip0.2cm
Since the final time $T^\e$ is diverging to $+\infty$ for $\e \to 0$, estimate \eqref{boundw} holds globally in time, and the precise decomposition for the perturbation $v$ can be used in the equation for $\xi(t)$ in order to decoupled the system \eqref{qulinsys}. Indeed
\begin{equation*}
|\langle \d_\xi \psi_1^\e, v \rangle| \leq C \left( |v_0|_{{}_{L^2}} e^{\Lambda_2^\e t}+ |\Omega^\e|_{{}_{L^\infty}} \right),
\end{equation*}
so that, for small $\e$ and $|v_0|_{{}_{L^2}}$, the function $\xi(t)$ behaves like the solution $\zeta(t)$ to the following problem
\begin{equation*}
\frac{d\zeta}{dt} = \theta^\e(\zeta), \quad \zeta(0)=\xi_0.
\end{equation*}

\begin{proposition}\label{prop:slowmotion}
Let hypotheses {\bf H1-2} be satisfied. Let us also assume that
\begin{equation}
	\theta^\varepsilon(\xi)<0\quad\textrm{ for any } \xi\in I,
	\qquad\textrm{ and }\qquad
	{\theta^\varepsilon}'(\bar \xi)<0.
\end{equation}
Then, for $\varepsilon$ and $|v_0|_{{}_{L^2}}$ sufficiently small, $\xi(t)$ converges to its equilibrium location $\bar \xi$ as $t\to+\infty$.
\end{proposition}

\begin{proof}
For any initial datum $\xi_0$, the variable $\xi(t)$ solves an equation of the form
\begin{equation*}
	\frac{d\xi}{dt}= \theta^\varepsilon(\xi)(1+r(t))
	\qquad \textrm{with}\quad 
	|r(t)|\leq C(|v_0|_{{}_{L^2}}e^{\Lambda_2^\varepsilon t}+|\Omega^\varepsilon|_{{}_{L^\infty}}).
\end{equation*}
Therefore, by means of the standard method of separation of variables  and since $\theta^\varepsilon(\xi) \sim {\theta^\varepsilon}'(\bar \xi) (\xi -\bar \xi) $, we get
 \begin{equation*}
 \int_{\xi_0}^{\xi(t)} \frac{d\eta}{\eta-\bar \xi} \sim  {\theta^\e}'(\bar \xi) \, t,
 \end{equation*}
that is, $\xi$ converges to $\bar \xi$ as $t\to+\infty$ and the convergence is exponential,
in the sense that there exists $\beta^\varepsilon>0$ such that
\begin{equation}\label{zetaestimate}
	\xi(t) \sim \bar \xi + \xi_0e^{-\beta^\varepsilon t}, \quad \beta^\e = -{\theta^\e}'(\bar \xi)
\end{equation}
for any $t$ under consideration.
\end{proof}

Estimate \eqref{zetaestimate} shows the exponentially slow motion of the interface
for small $\varepsilon$. Indeed, its evolution towards the equilibrium
position is much slower as $\varepsilon$ becomes smaller, since
$\beta^\varepsilon \to 0$ as $\varepsilon \to 0$.

More precisely, the location of the interface $\xi(t)$ remains close to some non equilibrium value for a time $T^\e$ that can be extremely long when $\e$ is small, before converging to its stable configuration, corresponding to one of the wall $x = 0$ or $x=\ell$.

\subsection{Nonlinear metastability for a convection-reaction-diffusion equation}

The general result presented in Theorem \ref{thm:metaL} concerns with the reduced system \eqref{qulinsys}, obtained by disregarding higher order terms in the variable $v$ but still keeping the nonlinear dependence on $\xi$. In the last part of this section we mean to analyze the complete system for the couple $(\xi,v)$, that is
\begin{equation}\label{CS}
 	\left\{\begin{aligned}
	\frac{d\xi}{dt}=\theta^\varepsilon(\xi)\bigl(1+\langle\partial_{\xi} \psi^\varepsilon_1, v \rangle\bigr)
		+ \rho^\varepsilon[\xi,v]), \\
	\partial_t v= H^\varepsilon(x;\xi)
		+ ({\mathcal L}^\varepsilon_\xi+{\mathcal M}^\varepsilon_\xi)v
			+\mathcal{R}^\varepsilon[v,\xi],
 	\end{aligned}\right. 
\end{equation}
where 
\begin{equation*}
\begin{aligned}
\rho^\e[\xi,v]&:= \langle \psi^\e_1, \mathcal Q^\e[v,\xi] \rangle \bigl(1+\langle\partial_{\xi} \psi^\varepsilon_1, v \rangle\bigr)+ \langle \psi^\e_1, \mathcal P^\e[U^\e]+ \mathcal Q^\e[v,\xi] \rangle R(v), \\\quad 
 \mathcal{R}^\varepsilon[v,\xi]&:=\mathcal{Q}^\varepsilon[v,\xi]
								-\partial_{\xi}U^{\varepsilon}(\cdot;\xi)\,\rho^\varepsilon[\xi,v].
								\end{aligned}
\end{equation*}
In particular, recalling that
\begin{equation*}
\mathcal Q^\e[v,\xi]:= \frac{1}{2} \left\{  -\d_x \left( f''(U^\e) v^2  \right)+ f'''(U^\e)v^2\right\}.
\end{equation*}
we have
\begin{equation*}
|\mathcal Q^\e[v,\xi]|_{L^1} \leq C |v|^2_{H^1}.
\end{equation*}
Thus, in order to proceed with the same methodology implemented in the proof of Theorem \ref{thm:metaL}, we also need an estimate for the $L^2$ norm of the space derivative of the component $v$.

\begin{theorem}\label{thm:metaNL}
Let hypotheses {\bf H1-2} be satisfied and let us denote by $(\xi,v)$ the solution to the initial-value problem \eqref{CS}, with
\begin{equation*}
\xi(0)=\xi_0 \ \in \ (0,\ell) \quad {\rm and} \quad v(x,0)=v_0(x) \ \in \ H^1(I),
\end{equation*}
Then,  for $\varepsilon$ sufficiently small, there exists a time $T^\varepsilon \geq 0$, such that, for any $t \leq T^\varepsilon$, the solution $v$ can be represented as
\begin{equation*}
v=z+R,
\end{equation*}
where $z$ is defined by
\begin{equation*}
	z(x,t):=\sum_{k\geq 2} v_k(0)\exp\left(\int_0^t \lambda^\varepsilon_k(\xi(\sigma))\,d\sigma\right)
		\,\varphi^\varepsilon_k(x;\xi(t)),
\end{equation*}
and the remainder $R$ satisfies the estimate
\begin{equation}\label{boundresto}
	|R|_{{}_{H^1}}\,\leq C\,
		\left\{ \varepsilon^\delta \, \exp\left( \int_0^t\lambda_1^\varepsilon(\xi(\sigma))d\sigma \right)|v_0|^2_{{}_{H^1}}+\varepsilon^{\gamma-\delta}\right\},
\end{equation}
for some constant $C>0$ and for some $\delta \in (0,\gamma)$, $\gamma>0$. Furthermore, the final time $T^\varepsilon$ can be chosen of order $1/\e^{\alpha}$, for some $\alpha>0$.
\end{theorem}
We note that the estimate \eqref{boundresto} is weaker that the corresponding formula \eqref{boundw} obtained for the reduced system, since it states that the remainder $R$ tends to $0$ as $\varepsilon^\delta$ instead of $|\Omega^\e|_{{}_{\infty}}$. Such deterioration of the estimate is a consequence of the necessity of estimating also the first order derivative and it is probably related to the specific strategy we use at such stage. We suspect that this bound is not optimal. Anyway, let us stress that this nonlinear result makes the theory much more complete.

\begin{proof}
Since the plan of the proof closely resemble the one used for proving Theorem \ref{thm:metaL}, we
propose here only the major modifications of the argument. In particular, the key point is how to handle the nonlinear terms.

 Setting as usual
\begin{equation*}
	v(x,t)=\sum_{j} v_j(t)\,\varphi^\varepsilon_j(x,\xi(t)),
\end{equation*}
we obtain an infinite-dimensional differential system for the coefficients $v_j$
\begin{equation}\label{eqwk_bis}
	\frac{dv_k}{dt}=\lambda^\varepsilon_k(\xi)\,v_k
		+\langle \psi^\varepsilon_k,F\rangle+ +\langle \psi^\varepsilon_k,G\rangle,
\end{equation}
where $F$ is defined as before as
\begin{equation*}
	F:=H^\varepsilon-\theta^\varepsilon \sum_{j}\Bigl(a_j+\sum_{\ell} b_{j\ell}\,v_\ell\Bigr)v_j,
\end{equation*}
with
\begin{equation*}
	a_j:=\langle \partial_{\xi} \psi^\varepsilon_1, \varphi^\varepsilon_j\rangle\,
			\partial_{\xi}U^{\varepsilon}+\partial_\xi \varphi^\varepsilon_j,
	\qquad
	b_{j\ell}:=\langle \partial_{\xi} \psi^\varepsilon_1, \varphi^\varepsilon_\ell\rangle
			\,\partial_\xi \varphi^\varepsilon_j.
\end{equation*}
The term $G$ comes out from the higher order terms $\rho^\varepsilon$ and $\mathcal R^\varepsilon$ and has the following expression
\begin{equation*}
G:= \mathcal Q^\varepsilon- \left(  \sum_j  \partial_\xi \varphi_j^\varepsilon v_j+\partial_\xi U^\varepsilon\right) \left\{ \frac{\langle \psi^\varepsilon_1,\mathcal Q^\varepsilon \rangle}{1-\langle\partial_\xi \psi^\varepsilon_1,v \rangle}-\theta^\varepsilon\frac{\langle \partial_\xi \psi^\varepsilon_1,v \rangle^2}{1-\langle\partial_\xi \psi^\varepsilon_1,v \rangle} \right\}.
\end{equation*}
Moreover, we have
\begin{equation*}
|\langle \psi^\varepsilon_k, G \rangle| \leq (1+|\Omega^\varepsilon|_{{}_{L^\infty}}) |v|^2_{{}_{L^2}}+ C |v|^2_{{}_{H^1}},
\end{equation*}
so that, setting
\begin{equation*}
	E_k(s,t):=\exp\left( \int_s^t \lambda_k^\varepsilon(\xi(\sigma))d\sigma\right)
\end{equation*}
and since $v_1=0$, we have the following expression for the coefficients $v_k$, $k\geq 2$
\begin{equation*}
v_k(t)= v_k(0) E_k(0,t)+ \int_0^t \left\{ \right \langle \psi^\varepsilon_k,F\rangle+ +\langle \psi^\varepsilon_k,G\rangle \} E_k(s,t) \, ds.
\end{equation*}
By introducing the function
\begin{equation*}
	z(x,t):=\sum_{k\geq 2} v_k(0)\,E_k(0,t)\,\varphi^\varepsilon_k(x;\xi(t)),
\end{equation*}
we end up with the following estimate for the $L^2$-norm of the difference $v-z$
\begin{equation*}
|v-z|_{{}_{L^2}} \leq \sum_{k \geq 2} \int_0^t \left( \Omega^\e(\xi)(1+|v|^2_{{}_{L^2}}) + |v|^2_{{}_{H^1}}\right) E_k(s,t) \, ds,
\end{equation*}
that is
\begin{equation*}
\begin{aligned}
|v-z|_{{}_{L^2}} \leq& \ C \int_0^t \Omega^\varepsilon(\xi) (t-s)^{-1/2} E_2(s,t)  ds \\
&+ \int_0^t \left( |v-z|^2_{{}_{H^1}}+ |z|^2_{{}_{H^1}}\right) (t-s)^{-1/2} E_2(s,t)  ds,
\end{aligned}
\end{equation*}
where we used
 \begin{equation*}
	\sum_{k\geq 2} E_k(s,t)
		\leq C\,(t-s)^{-1/2}\,E_2(s,t).
\end{equation*}
Now we need to differentiate with respect to $x$ the equation for $v$ in order to obtain an estimate for $|\partial_x(v-z)|_{{}_{L^2}}$. By setting $y=\partial_x v$, we obtain
\begin{equation*}
\partial_t y= \mathcal L^\varepsilon_\xi y + \bar{\mathcal  M}^\varepsilon_\xi v-\partial_x\left( \partial_xU^\varepsilon \, v\right) + \partial_xH^\varepsilon(x,\xi)+\partial_x \mathcal R^\varepsilon[\xi,v],
\end{equation*}
where
\begin{equation*}
\bar{\mathcal M}^\varepsilon_\xi v:=-\partial_{\xi x}U^{\varepsilon}(\cdot;\xi)
			\,\theta^\varepsilon(\xi)\,\langle\partial_{\xi} \psi^\varepsilon_1, v \rangle.
\end{equation*}
Hence, by setting as usual
\begin{equation*}
y(x,t)=\sum_j y_j(t) \, \varphi_j^\varepsilon (x,\xi(t)),
\end{equation*}
we have
\begin{equation*}
\frac{dy_k}{dt}=\lambda^\varepsilon_k(\xi)\,y_k
		+\langle \psi^\varepsilon_k,F^*\rangle-\langle \psi^\varepsilon_k ,\partial_x\left( \partial_x U^\varepsilon\,v\right)\rangle +\langle \psi^\varepsilon_k,\partial_x \mathcal R^\varepsilon \rangle,
\end{equation*}
where
\begin{equation*}
F^*:= \partial_x H^\varepsilon -\sum_j v_j \left \{ \theta^\varepsilon\left[ \partial_{\xi x} U^\varepsilon \langle \partial_\xi \psi^\varepsilon_1,\varphi_j^\varepsilon \rangle + \partial_\xi \varphi^\varepsilon_j\left( 1 + \sum_{\ell}v_{\ell} \langle \partial_\xi \psi^\varepsilon_1,\varphi_{\ell} \rangle \right)   \right] -\partial_\xi \varphi^\varepsilon_j \rho^\varepsilon\right \}.
\end{equation*}
Moreover, for some $m >0$, there hold
\begin{equation*}
|\langle \psi^\varepsilon_k, \partial_x \mathcal R^\varepsilon \rangle| \leq C |v|^2_{{}_{H^1}}, \quad |\langle \psi^\varepsilon_k ,\partial_x\left( \partial_x U^\varepsilon \, v\right)\rangle| \leq \varepsilon^m\, |U^\varepsilon|_{{}_{L^\infty}}^2+ \frac{1}{\varepsilon^m}\,|v|^2_{{}_{H^1}}.
\end{equation*}
Again, by integrating in time and by summing on $k$, we end up with
\begin{equation*}
\begin{aligned}
|y-\partial_x z|_{{}_{L^2}} \leq&\, C \int_0^t \left \{ \Omega^\varepsilon(\xi) (1+ |v|^2_{{}_{L^2}}) +\left(1+\frac{1}{\varepsilon^m}\right)\, |v|^2_{{}_{H^1}} + \varepsilon^m \, |U^\varepsilon|^2_{{}_{L^\infty}} \right \} \, E_1(s,t) ds \\
&+C \int_0^t \left \{ \Omega^\varepsilon(\xi) (1+ |v|^2_{{}_{L^2}}) +\left(1+\frac{1}{\varepsilon^m}\right)\, |v|^2_{{}_{H^1}} + \varepsilon^m \, |U^\varepsilon|^2_{{}_{L^\infty}} \right \} \, \sum_{k \geq 2} E_k(s,t) ds.
\end{aligned}
\end{equation*}
Now, given $n>0$, let us set
\begin{equation*}
	N(t):= \frac{1}{\varepsilon^n} \sup_{s\in[0,t]} |v-z|_{{}_{H^1}}\,E_1(s,0),
\end{equation*} 
so that we have
\begin{equation}\label{est1}
\begin{aligned}
\frac{1}{\varepsilon^n}E_1(t,0)|v-z|_{{}_{L^2}} \leq& \ C \int_0^t \frac{\Omega^\varepsilon(\xi)}{\varepsilon^n} (t-s)^{-1/2} E_2(s,t) E_1(s,0) ds \\
&+ \int_0^t \frac{1}{\varepsilon^n}\left( |v-z|^2_{{}_{H^1}}+ |z|^2_{{}_{H^1}}\right) (t-s)^{-1/2} E_2(s,t) E_1(s,0) ds
\end{aligned}
\end{equation}
and
\begin{equation}\label{est2}
\begin{aligned}
&\frac{1}{\varepsilon^n}E_1(t,0)|y-\partial_x z|_{{}_{L^2}} \leq  
C \int_0^t \frac{\Omega^\varepsilon(\xi)}{\varepsilon^n}  \left \{E_1(s,0)+(t-s)^{-1/2} E_s(s,t)\, E_1(s,0)\right \} \,  ds\\
& +C \int_0^t \left \{ \left(\frac{1}{\e^n}+\frac{1}{\e^{n+m} } \right)\,\left( |v-z|^2_{{}_{H^1}} + |z|^2_{{}_{H^1}}\right)+\frac{1}{\varepsilon^{n-m}} \, |U^\varepsilon|^2_{{}_{L^\infty}} \right \} \, E_1(s,0) ds \\
& +C \int_0^t\left \{ \left(\frac{1}{\e^n}+\frac{1}{\e^{n+m}} \right)\!\!\left( |v-z|^2_{{}_{H^1}} \!+\! |z|^2_{{}_{H^1}}\right)\!+\!\frac{1}{\e^{n-m}} \, |U^\varepsilon|^2_{{}_{L^\infty}} \right \} (t-s)^{-1/2} E_s(s,t) E_1(s,0) ds. \\
\end{aligned}
\end{equation}
By summing \eqref{est1} and \eqref{est2} and since there hold
\begin{align*}
	&\int_0^t e^{(2\Lambda^\varepsilon_2-\Lambda_1^\varepsilon)s}\,ds
		\leq \int_0^t e^{\Lambda^\varepsilon_2 s}\,ds
		=\frac{1}{\Lambda_2^\varepsilon}(e^{\Lambda^\varepsilon_2 s}-1)\leq  \frac{1}{|\Lambda_2^\varepsilon|}, 
		\\
	&\int_0^t (t-s)^{-1/2}\,E_2(s,t)\,ds
		\leq \int_0^t (t-s)^{-1/2}\,e^{\Lambda_2^\varepsilon\,(t-s)}\,ds
		\leq \frac{1}{|\Lambda_2^\varepsilon|^{1/2}}, 
\end{align*}
we end up with the estimate $N(t) \leq A N^2(t)+B$, with
 \begin{equation*}
		\left\{\begin{aligned}
		A&:= \varepsilon^{n-m} \, E_1(0,t) (t + |\Lambda_2^\varepsilon|^{-1/2}) ,\\
		B&:=C|\Omega^\varepsilon|_{{}_{L^\infty}}E_1(t,0) \left( t + |\Lambda_2^\varepsilon|^{-1/2} \right) \\ & \quad + \varepsilon^{-n-m} |\Lambda_2^\varepsilon|^{-1} |v_0|^2_{{}_{H^1}}+ \varepsilon^{m-n} |U^\varepsilon|^2_{{}_{L^\infty}} E_1(t,0)\, (t +|\Lambda_2^\varepsilon|^{-1/2}).
		\end{aligned}\right.
\end{equation*}
Supposing that $|\Lambda_2^\varepsilon| \sim \e^{-\gamma}$ for some $\gamma >0$, if we require $m < \gamma$, for all $n>0$ there holds $N(t)< 2B$ that is
\begin{equation}\label{finnonlin}
|v-z|_{{}_{H^1}} \leq C |\Omega^\varepsilon|_{{}_{L^\infty}} + \Bigl( \varepsilon^{\gamma-m} |v_0|^2_{{}_{H^1}}E_1(0,t)+\e^m\Bigr).
\end{equation}
Precisely,  we can choose $m=\gamma-\delta$, for some $\delta \in (0,\gamma)$. Finally, providing  $m>n$, we can choose the final time $T^\varepsilon$ of order $\mathcal O(\e^{-\alpha})$ for some $0<\alpha<1$. Now the proof is completed.

\end{proof}

Estimate \eqref{finnonlin} can now be used to decouple the complete system \eqref{CS} more precisely, under the hypotheses of Proposition \ref{prop:slowmotion} on the function $\theta^\varepsilon(\xi)$, we get
\begin{equation*}
\frac{d\xi}{dt}= \theta^\varepsilon(\xi)(1+r)+ \rho^\varepsilon,
\end{equation*}
where
\begin{equation*}
\begin{aligned}
|r| \leq |v_0|^2_{{}_{H^1}} e^{-c\,t} + (\varepsilon^\delta + \varepsilon^{\gamma-\delta}) \quad {\rm and} \quad |\rho^\varepsilon| \leq C\, |v|_{{}_{H^1}}^2 \leq |v_0|^2_{{}_{H^1}}e^{-ct} + \varepsilon^\delta + \varepsilon^{\gamma-\delta}.
\end{aligned}
\end{equation*}
Hence,  there exists $\beta^\varepsilon >0$, $\beta^\varepsilon \to 0$ as $\varepsilon \to 0$, such that
\begin{equation*}
|\xi-\bar \xi| \leq |\xi_0|e^{-\beta^\varepsilon t} 
+ \left \{  |v_0|^2_{{}_{H^1}}e^{-ct} + \varepsilon^\delta + \varepsilon^{\gamma-\delta} \right\} e^{-\beta^\varepsilon t},
\end{equation*}
showing the exponentially (slow) motion of the position of the interface towards its equilibrium location $\bar \xi$.

\section{Application to the Burgers-Sivashinsky equation}\label{Sec:example}

The aim of this Section it to apply the general theory developed in the previous Sections to the specific example of the so called Burgers-Sivashinsky equation, that is
\begin{equation}\label{BSexample}
\d_t u= \e \d^2_x u -u \d_x u + u,
\end{equation}
complemented with boundary conditions and initial datum
\begin{equation}\label{IBV}
u(0,t)= u (\ell,t)=0 \ \ \  t \geq 0, \quad {\rm and} \quad u(x,0)=u_0(x) \ \ \ x \in I.
\end{equation}
More precisely, we mean to show that the hypotheses of Theorem \ref{thm:metaL} and Theorem \ref{thm:metaNL} are satisfied in this specific case.
\vskip0.2cm
To begin with, let us consider the case $\e=0$: equation \eqref{BSexample} formally reduces to the first order hyperbolic equation 
\begin{equation}\label{BShyp}
\d_t u=-u\d_x u+ u,
\end{equation}
together with boundary and initial conditions. In this case, stationary solutions solve the first order equation
\begin{equation*}
\frac{d}{dx} \left( \frac{u^2}{2} \right) = u \quad \Longrightarrow \quad \frac{du}{dx}=1,
\end{equation*}
where the boundary conditions have to be interpreted in the sense of \cite{BardLeRoNede79}.
We concentrate on entropy stationary solutions with at most one internal jump. Hence, we can construct a two parameters family of steady states, solutions to \eqref{BShyp}, defined as
\begin{equation*}
U_{{}_{\rm hyp}}(x):=\left\{\begin{aligned}
&x-\xi_1 &\quad 0 \leq & \ x < (\xi_1+\xi_2)/2 \\
&x-\xi_2 &\quad    (\xi_1+\xi_2)/2 < & \ x \leq \ell
\end{aligned}\right.
\end{equation*}
where the parameters $(\xi_1,\xi_2) $ belong to the triangle $T:=\{ (\xi_1,\xi_2) \ : \ 0 \leq \xi_1 \leq \xi_2 \leq \ell \}$. In particular, it is possible to distinguish several different regions of the triangle, corresponding to different types of steady states with different properties.

\vskip0.1cm
{\bf 1.} When $(\xi_1,\xi_2)$ belongs to the interior of the triangle $ \bar T:=\{ (\xi_1,\xi_2) \ : \ 0 < \xi_1 < \xi_2 < \ell \}$, the corresponding solution has two boundary layers and a single jump inside the interval $(0,\ell)$.
\vskip0.2cm
{\bf 2.} When $(\xi_1,\xi_2)$ belong to one of the sides $\G_1 := \{ (\xi_1,\xi_2) \ : \ \xi_1=0, \ 0 < \xi_2 < \ell\}$ and $\G_2:= \{ (\xi_1,\xi_2) \ : \  0 < \xi_1 < \ell, \ \xi_2=\ell\}$, the corresponding steady state have one internal jump and one boundary layer.
\vskip0.2cm
{\bf 3.} When $(\xi_1,\xi_2)$ belongs to the diagonal side $D:= \{ (\xi_1,\xi_2) \subset T \ : \  \xi_1=\xi_2 \}$, the steady $U_{{}_{\rm hyp}}$ has two boundary layers an no jump discontinuities. It is the hyperbolic version of $U^\e_{{}_{M}}$, and we named here $U^0_{{}_{M}}$.
\vskip0.2cm
{\bf 4.} When $(\xi_1,\xi_2)=(0,\ell)$, the steady state has a single internal jump in the middle point of the interval, and no boundary layers.  We call it here $U^0_{{}_{NS}}$,  the rougher version of $U^\e_{{}_{NS}}$.

\vskip0.2cm
{\bf 5.} Finally, when $(\xi_1,\xi_2)$ coincide with one of the vertex $(0,0)$ or $(\ell,\ell)$, the corresponding solutions have one boundary layer (at the right side and at the left side of the interval respectively), and no jump discontinuities. Such solutions are the hyperbolic version of $U_{\e,+}$ and $U_{\e,-}$ respectively, and we refer to them as $U_{0,\pm}$.
\vskip0.2cm
\noindent Hence, at the level $\e =0$, we have infinitely many steady states with different properties (see Figures \ref{fig4} and \ref{fig5}). 
\begin{figure}
\centering
\includegraphics[width=6.5cm,height=6cm]{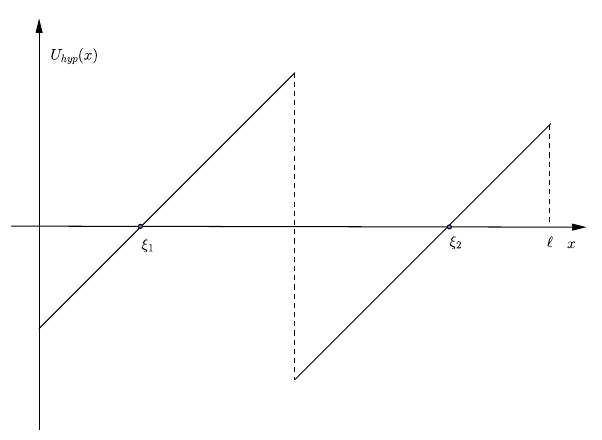}
 \hspace{3mm}
\includegraphics[width=6.5cm,height=6cm]{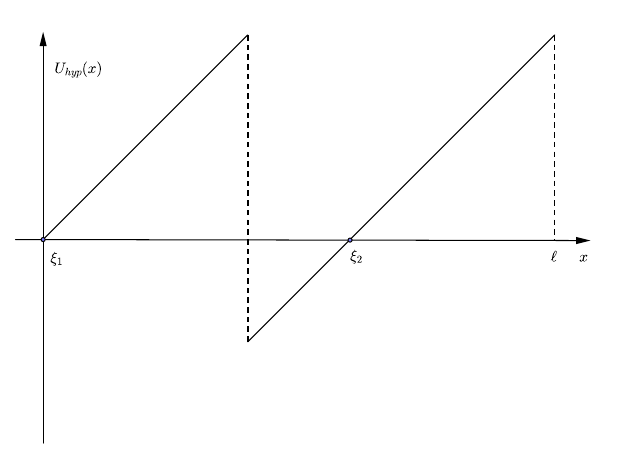}
\caption{\small{Stationary solutions to $\d_t u=-u \d_x u+u$ corresponding  to the cases {\bf 1 -2} described above.}}\label{fig4}
\end{figure}
 \begin{figure}
\includegraphics[width=7cm,height=6cm]{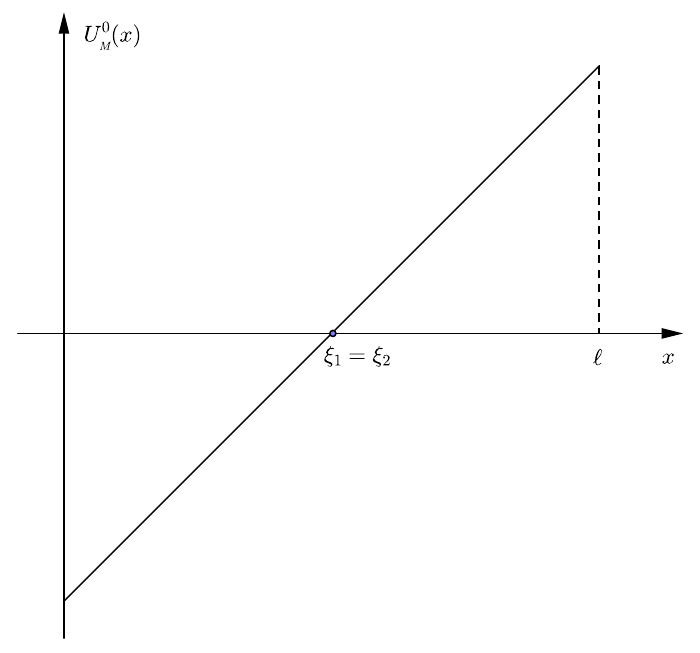}
 \hspace{3mm}
\includegraphics[width=7cm,height=6cm]{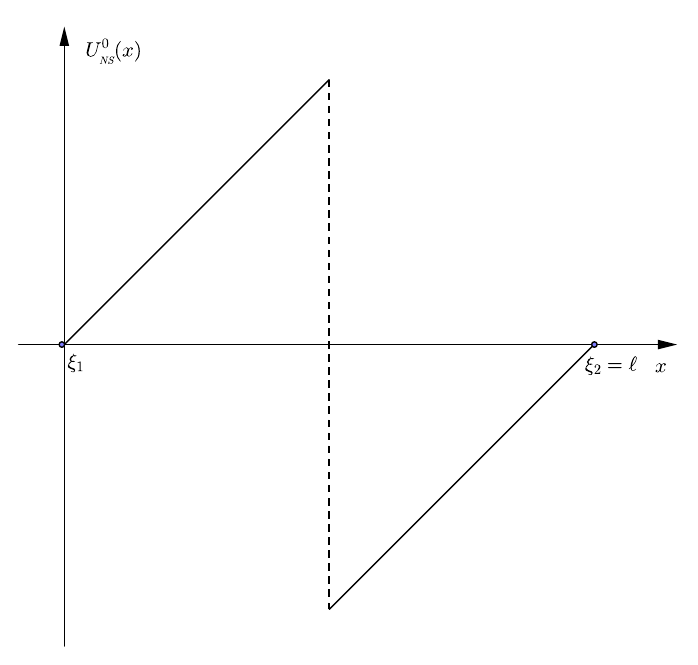}
\hspace{5mm}
\includegraphics[width=10cm,height=6cm]{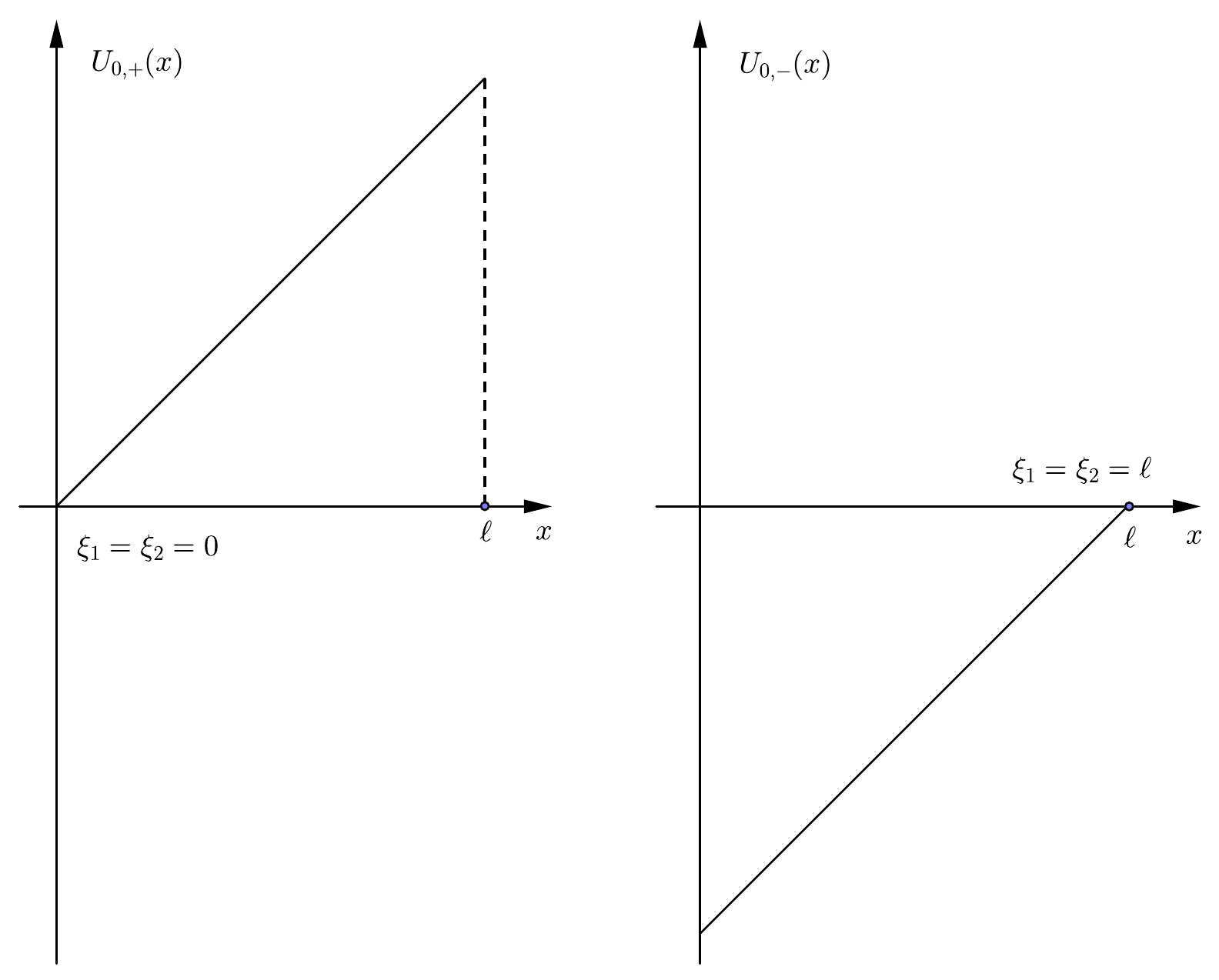}
\caption{\small{Stationary solutions to $\d_t u=-u \d_x u+u$ for different choices of the parameters $\xi_1, \xi_2 \in T$; the pictures correspond, respectively, to $U^0_{{}_{M}}$, $U^0_{{}_{NS}}$ and  $U_{0,\pm}$. In particular, $U^0_{{}_{M}}$ is the hyperbolic version of the metastable steady state $U^\e_{{}_{M}}$.}}\label{fig5}
\end{figure}

\vskip0.2cm
When $\e>0$, the presence of the Laplace operator has a smoothing effect on stationary solutions: all the jump discontinuities turn out into a viscous shock. It is possible to obtain an implicit expression for the stationary solutions to \eqref{BSexample}, that solve
\begin{equation*}
\frac{d}{dx}\left( \frac{u^2}{2}-\e\frac{du}{dx}\right)=u \ \ \ \Longleftrightarrow \quad
 \left\{ \begin{aligned}    
 \e \frac{du}{dx}  &= \frac{u^2}{2}-v \\
\frac{dv}{dx}&= u.
 \end{aligned}\right.
\end{equation*}
If we look for a solution of the form $u=\sigma(v)$, we get the following Bernoulli equation
\begin{equation*}
\frac{d\sigma}{dv}= \frac{du}{dx} \frac{dx}{dv}=\frac{1}{2\e} \sigma-\frac{1}{\e \sigma} v,
\end{equation*}
that can be solved by means of the standard change of variable $\omega=\sigma^2$. We deduce the following equation for $\omega$
\begin{equation}\label{ustazimpl}
\frac{d\omega}{dv}- \frac{\omega}{\e}= -\frac{2v(x)}{\e} \ \ \ \Longleftrightarrow \ \ \ \frac{d}{dv} \left( e^{-v/\e} \omega\right)= \frac{d}{dv} \left( 2(\e+v) e^{-v/\e}\right),
\end{equation}
whose solution is given by
\begin{equation*}
\frac{1}{2} u^2= \e + v+ \kappa \, e^{v/\e}, \quad \frac{dv}{dx}=u
\end{equation*}
for some $\kappa \in \R$. 
We are interested in stationary solutions corresponding, when $\e >0$, to the steady states $U_{\e,\pm}(x)$ (case {\bf 5}), $U^\e_{{}_{NS}}(x)$ (case {\bf 4}) and, specially,  $U^\e_{{}_{M}}(x)$ (case {\bf 3}), the one that gives rise to a metastable behavior. In this cases, stationary solutions to \eqref{BSexample} can be seen as a smoothed version  of the states $U_{0,\pm}$, $U^0_{{}_{NS}}$ and of the state $U^0_{{}_{M}}$ with the choice $\xi_1=\xi_2=\ell/2$.

\vskip0.3cm
\subsection{The family of approximate steady states and the remainder $\mathcal P^\e[U^\e]$}
Following the general approach introduced in Section \ref{sect:general}, the first step is the construction of a one parameter family of functions $\{U^\e(x;\xi)\}_{\xi \in I}$ that approximate the metastable steady state $U^\e_{{}_{M}}(x)$. There are several possible choices to build up such family. One idea is to define $U^\e(x;\xi)$ as a smoothed version of the hyperbolic steady state $U^0_{{}_{M}}(x)$. The smoothing consists in substituting a viscous shock-like version of the layers in a small neighborhood of the transition point, based on the explicit formula
\begin{equation*}
U^\e_{{}_{\rm vsc}}(x):= - u_{*} \ \tanh \left( \frac{u_*(x-a_0)}{2\e}\right),
\end{equation*}
that describes a viscous shock wave centered in $a_0$ connecting the left state $u_*$ and the right state $-u_*$. Hence we define
\begin{equation}\label{approx1}
U^\e(x;\xi) = \max \left\{ \min \left\{ x-\xi , -(\ell-\xi) \tanh \left(\frac{\ell-\xi}{2\e} \, (x-\ell) \right) \right\}, -\xi \tanh \left( \frac{\xi}{2\e} \, x\right) \right \},
\end{equation}
where the parameter $\xi \in I$ represents the unique zero of the function $U^\e$ (see Figure \ref{fig5}). 

With such a construction, $ U^\e $ matches exactly the boundary conditions and satisfies the stationary equation up to an error, denoted here as $\mathcal P[U^\e]$. 

Moreover, for fixed $\xi \in I$, each element of the family $\{ U^\e(x;\xi) \}$ converges, as $\e \to 0$, to $U^0_{{}_{M}}(x)$, defined as
\begin{equation*}
U^0_{{}_{M}}(x)= x-\xi, \ \ \ \xi \in (0,\ell), \quad U^0_{{}_{M}}(0)=U^0_{{}_{M}}(\ell)=0.
\end{equation*}

Once the family of approximate steady states is explicitly given, next step is to compute the error $\mathcal P^\e[U^\e]$. We recall here that, in order to compute the remainder $\mathcal P^\e[U^\e]$, we need a explicit expression for the term $\Omega^\e(\xi)$, defined as
\begin{equation*}
|\langle \psi(\cdot ), \mathcal P^\e[U^\e(\cdot;\xi)]\rangle| \leq \Omega^\e(\xi)|\psi|_{L^\infty}, \quad \forall \, \psi \in C(I), \, \forall \, \xi \in I.
\end{equation*}
Given $U^\e$ as in \eqref{approx1}, we check that $U^\e$ is an exact stationary solution for \eqref{BSexample}, outside the intervals $(0,u_1^*)$ and $(u_2^*,\ell)$, where $u_1^*$ and $u_2^*$ are implicitly defined as
\begin{equation}\label{lauso}
\begin{aligned}
-\xi \tanh \left( \frac{\xi}{2\e} \, u_1^*\right)=u_1^*-\xi, \qquad
-(\ell-\xi) \tanh \left(\frac{\ell-\xi}{2\e} \, (u_2^*-\ell) \right)=u_2^*-\xi.
\end{aligned}
\end{equation} 
From \eqref{lauso}, since $u_1^* \sim 0$ and $u_2^* \sim \ell$, we get
\begin{equation*}
u_1^* \sim \varepsilon \, \xi \ \ \ {\rm and} \ \ \ u_2^* \sim \ell -\varepsilon \, \xi \ \ \ \ \ {\rm as} \ \ \ \e \to 0,
\end{equation*}
where we used the Taylor expansion of $\tanh (x)$ for $x \sim 0$.
\begin{figure}
\centering
\includegraphics[width=12cm,height=7cm]{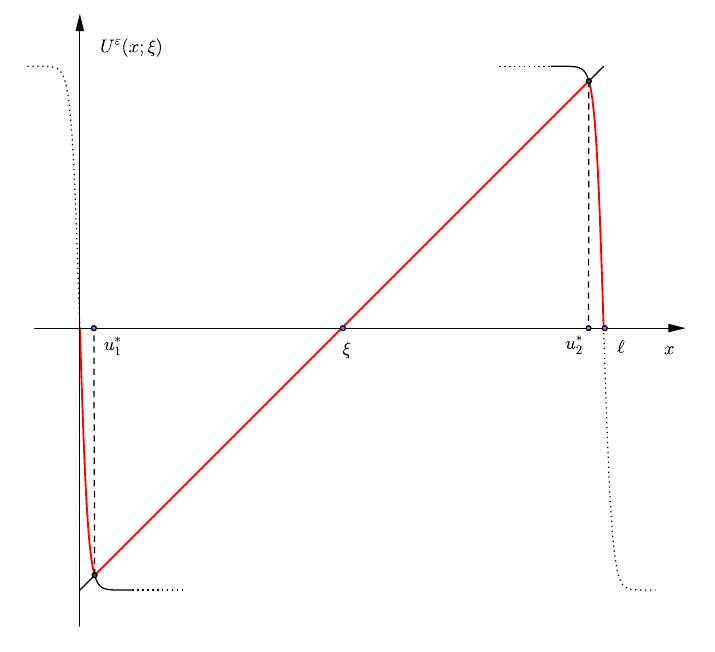}
\caption{\small{The element $U^\e$ of the family of approximate stationary solutions obtained as a smoothed version of the hyperbolic steady state $U^0_M(x)$.}}\label{fig5}
\end{figure}
Now we can derive an asymptotic expression for the term $\Omega^\e(\xi)$. Indeed we get
\begin{equation*}
\begin{aligned}
|\langle \psi(\cdot ), \mathcal P^\e[U^\e(\cdot;\xi)]\rangle| &\leq |\psi|_{{}_{L^\infty}} \left( \int_0^{u_1^*} \left | \xi \tanh \left( \frac{\xi}{2\e} \, x\right) \right |  \, dx  + \int_{u^*_2}^\ell \left|   (\ell-\xi) \tanh \left(\frac{\ell-\xi}{2\e} \, (x-\ell) \right) \right| \, dx\right) \\
&\leq  |\psi|_{{}_{L^\infty}}  |U^\e|_{{}_{L^1(0,u_1^*)}} + |\psi|_{{}_{L^\infty}}  |U^\e|_{{}_{L^1(u_2^*,\ell)}} \\
& \leq |\psi|_{L^\infty} \left[ \xi  \, u_1^*+ (\ell-\xi) (\ell-u_2^*) \right],
\end{aligned}
\end{equation*}
and we deduce
\begin{equation*}
\Omega^\varepsilon(\xi) \sim \xi^2 \, \varepsilon + (\ell -\xi) \, \xi \, \varepsilon,
\end{equation*}
showing that $\Omega^\e$ is null at $\xi=0$ and small with respect to $\e$. We can also define 
\begin{equation*}
\omega^\e(\xi) \sim \xi \, \varepsilon + (\ell -\xi) \, \varepsilon,
\end{equation*}
where $ \Omega^\e(\xi) \leq \omega^\e(\xi) \xi$, meaning that, when $ \xi=0$, the element $U^\e(x;0)$ solves the stationary equation. Indeed $U^\e(x;0)$,  defined as
\begin{equation*}
U^{\e}(x;0) := \min \left\{ x , -\ell \tanh \left(\frac{\ell}{\e}\, (x-\ell) \right) \right\} 
\end{equation*}
is an approximation (up to an error that is  small in $\e$) for the stable stationary solution $U_{\e,+}(x)$. 

\vskip0.2cm
It is also possible to derive an explicit expression for the leading order term in the equation of motion for $\xi(t)$, defined as
\begin{equation*}
 \theta^\e(\xi)=\frac{\langle \psi^\varepsilon_1,{\mathcal P^\e[U^{\varepsilon}] \rangle}}{\langle \psi^\e_1, \d_\xi U^\e\rangle}.
 \end{equation*}
Indeed, with the approximation $U^\e(x,\xi) \sim U_{{}_{hyp}}(x;\xi) =x-\xi$, we get $\partial_\xi U^\varepsilon =-1$, so that
\begin{equation*}
\langle \psi^\e_1,P^\e[U^\e]\rangle \sim (\xi^2 \, \varepsilon + (\ell -\xi) \, \xi \, \varepsilon) \, \int_{I} \psi_1^\e(x;\xi) \, dx \quad {\rm and } \quad  \langle \psi^\e_1, \d_\xi U^\e \rangle \sim -\int_I \psi^\e_1(x;\xi )\,dx,
\end{equation*}
and
\begin{equation}\label{thetaasy}
\theta^\e(\xi) \sim -\xi^2 \, \varepsilon - (\ell -\xi) \, \xi \, \varepsilon.
\end{equation}
Since, formally, for small $\varepsilon$ and small $v$, the dynamics of the parameter $\xi$ is
approximately given by 
\begin{equation*}
	\frac{d\xi}{dt}=\theta^\varepsilon(\xi), 
\end{equation*}
formula \eqref{thetaasy} shows that the speed of the interface is small with respect to $\e$. Finally,  $ \theta^\e(\xi) <0$ for all $\xi \in I$, and it is easy to check that
\begin{equation*}
{\theta^\e}'(0) = -\varepsilon \,\ell<0.
\end{equation*}
Hence, from \eqref{zetaestimate} (see Proposition \ref{prop:slowmotion}), we get
\begin{equation*}
 \xi(t) \sim \xi_0 e^{-\beta^\e t}, \quad {\rm with} \quad \beta^\e \sim \varepsilon,
\end{equation*}
showing the exponentially slow convergence of the interface location toward its equilibrium $\xi=0$. More precisely, the speed rate of convergence is proportional to $\varepsilon$, and the convergence is slower as $\e$ becomes smaller.
\begin{remark}{ \rm
The previous construction for the family of functions $\{ U^\varepsilon(x;\xi) \}$ is useful since we can derive an explicit expression for an element of the family, so that we can explicitly develop computations; also, it is possible to generalize this construction for a general nonlinearity $f$, by using in \eqref{approx1} the solution to $\e \d^2_{x} u = \d_x f(u)$ (instead of $U^\e_{{}_{\rm vsc}}$, solution to $\e \d^2_{x} u = u \d_x u$).
 However, this construction does not lead to an optimal estimate for the terms $\Omega^\varepsilon$ and $\theta^\varepsilon$, since we expect they behave like $e^{-1/\varepsilon}$.
}
\end{remark}

\vskip0.2cm
Following the line of \cite{MS}, another possible way to build up the family $\{ U^\e(x;\xi)\}$ consists in matching two exact steady states in the intervals $(0,\xi)$ and $(\xi,\ell)$ at $x=\xi$, under appropriate boundary conditions.  With such a construction $U^\e\in C^0(I)$ but it is not a $C^1$-matched function. In this case we have no an explicit expression for an element of the family of approximate steady states, and we can only state that the error $\mathcal P^\e[U^\e]$ is concentrated in $x=\xi$, and it is defined as
\begin{equation*}
\mathcal P^\e[U^\e(x;\xi)] = [\![ \partial_x U^\e]\!]_{x=\xi} \delta_{x=\xi},
\end{equation*}
so that
$
\Omega^\e(\xi) \sim [\![ \partial_x U^\e]\!]_{x=\xi}.
$
Motivated by the previous papers \cite{MS, MS2, Str12}, we expect this term to be exponentially small in $\e$, even if we are not able to prove it in this present paper. 

This possible construction is meaningful since it propose a way to build up a family of approximate steady states for any choice of $f$ that should lead to optimal estimates for the error terms, in contrast with the one given in \eqref{approx1}; indeed, with such a construction, the specific form of the nonlinearity $f$ is really taken into account when solving the viscous stationary equation in order to define $U^\e(x;\xi)$.  The main difficulty here is to be able to have an explicit expression for the exact steady state: in principle, this could be done numerically, and the computation of the error $\Omega^\e$ should follow easily.

\subsection{The spectrum of the linearized operator} We analyze the spectrum of the linearized operator $\mathcal L^\e_\xi$, and we give a precise distribution of its ({\bf real}) eigenvalues, in order to show that the general technique developed in Section \ref{sect:general} is indeed applicable in the case of the Burgers-Sivashinsky equations. 

Chosen an approximate steady state $U^\e(x;\xi)$ satisfying the boundary conditions $U^\e(0)=U^\e(\ell)=0$, and linearizing the equation \eqref{BSexample} around $U^\e$, we end up with the following linearized problem
\begin{equation}\label{opeesempio}
\partial_t v=\mathcal L^\e_\xi v:= \e \d_x^2 v- \d_x (U^\e v)+v, \quad v(0)=v(\ell)=0.
\end{equation}
First of all, let us notice that $\lambda^\e$ is an eigenvalue for $\mathcal L^\e_\xi$ if and only if $\lambda^\e-1$ is an eigenvalue for the differential linear diffusion-transport operator
\begin{equation}\label{operidotto}
\mathcal L^{\e,vsc}_\xi v := \e \d_x^2 v- \d_x (U^\e v),
\end{equation}
so that we reduce to the study of the eigenvalue problem for $\mathcal L^{\e,vsc}_\xi$. We recall here that $U^\e(x) \to U_{hyp}(x)$ for $\e \to 0$, where $U_{hyp}$ is defined as $U_{hyp}(x)=x-\xi$. More precisely there exists $C>0$ such that
\begin{equation*}
|U^\e-U_{hyp}|_{{}_{L^1}} \leq C \, \e.
\end{equation*}

\subsubsection{Estimate for the first eigenvalue} An asymptotic expression for the first eigenvalue of the linearized operator $\mathcal L^\e_\xi$ has been furnished by X. Sun and M. Ward in \cite{SunWard99}.  Here the authors linearize the equation \eqref{BSexample} around an approximate steady state obtained by using the so called method of matched asymptotic expansion; moreover they show that the principal eigenvalue associated with this linearization is positive and has the following asymptotic expression (for more details, see \cite[Section 3]{SunWard99})
\begin{equation*}
\lambda_1^\e(\xi) \sim \frac{1}{\e} \left [  \xi \, \left(\xi- c\,  \e^{1/2}\right) e^{-\xi^2/2\e}+ (\ell-\xi)\left((\ell-\xi)-c \, \e^{1/2}\right)  e^{-(\ell-\xi)^2/2\e}\right].
\end{equation*} 
In particular, the equivalence holds true in the limit $\e \to 0$. Positivity of $\lambda_1^\e$ is strictly related to the instability property of the stationary solution $U^\e_{{}_{M}}(x)$. However, even if $\lambda_1^\e$ is positive, it is exponentially small with respect to $\e$, leading to a metastable behavior for solutions starting from an initial datum close to $U^\e_{{}_{M}}$. Indeed the dynamics takes place within a time scale of order $e^{1/\varepsilon}$, since the long time behavior of solutions is described by terms of order $e^{\lambda_1^\varepsilon t}$.

\subsubsection{Estimate for the second eigenvalue} We mean to control from above the location of the second (and subsequent) eigenvalue of $\mathcal L^{\e,vsc}_\xi$, in order to control the location of the eigenvalues of $\mathcal L^\e_\xi$.  Let us introduce the self-adjoint operator
\begin{equation}\label{SAope}
\mathcal M^\e_{\xi} v := \e^2 \d_x^2 v - a^\e(x;\xi(t)) v, \quad a^\e(x;\xi(t)):= \left(\frac{U^\e}{2}\right)^2+\frac{1}{2} \, \e \, \d_x U^\e.
\end{equation}
It is easy to check that if $\varphi^\e$ is an eigenfunction of  \eqref{operidotto} of eigenvalue $\lambda^\e$, then the function $\psi^\e$ defined as
\begin{equation}\label{collauto}
\psi^\e(x;\xi)= \exp\left( -\frac{1}{2\e} \int_{x_0}^x U^\e(y) dy \right) \varphi^\e(x;\xi)
\end{equation}
is an eigenfunction of \eqref{SAope} relative to the eigenvalue $\mu^\e=\e\lambda^\e$. 

\vskip0.2cm
Now let us consider the couples $(\varphi^\e_2,\lambda^\e_2)$ and $(\psi^\e_2,\mu^\e_2)$. Since $\lambda_2^\e$ is the second eigenvalue, then there exists $x_0 \in (0,\ell)$ such that  $\varphi^\e_2(x_0)=0$. Hence, $\varphi^\e_2$ restricted to $(x_0,\ell)$ can be seen as the first eigenfunction  (relative to the first eigenvalue) of the linearized operator $\mathcal L^{\e,vsc}_\xi$ in the interval $(x_0,\ell)$ and with Dirichlet boundary conditions. In other words, $\varphi^\e_2(x)$ can be seen as the first eigenfuction of the linearized operator obtained by linearizing \eqref{BSexample} around $U^{\e,+}(x,x_0,\xi)$, that is an approximation of the unique positive stationary solution to \eqref{BSexample} in the interval $(x_0,\ell)$. The same argument holds for $\psi_2^\e(x)$.

From now on, without loss of generality, we restrict our analysis to the interval $J=(x_0,\ell)$ and we renormalize $\psi_2^\e$ so that $\max \psi_2^\e =1$. Integrating over $J$ we obtain
\begin{equation}\label{disuglambda2}
\lambda^\e_2 \int_{x_0}^\ell {\varphi^\e_2} dx = \e ({\varphi^\e_2}'(\ell)-{\varphi^\e_2}'(x_0)) < -\e {\varphi^\e_2}'(x_0).
\end{equation}
Moreover, because of \eqref{collauto} we have
\begin{equation*}
{\varphi^\e_2}'(x)= {\psi^\e_2}'(x)\exp\left( \frac{1}{2\e} \int_{x_0}^x U^\e(y) dy \right)+ \psi^\e_2(x) \exp\left( \frac{1}{2\e} \int_{x_0}^x U^\e(y) dy \right) \frac{U^\e(x)}{2\e},
\end{equation*}
so that, since $\psi^\e_2(x_0)=0$, we get ${\varphi^\e_2}'(x_0)={\psi^\e_2}'(x_0)$. Hence, from \eqref{disuglambda2} we can state that
\begin{equation*}
|\lambda_2^\e| > \e I_\e^{-1} {\psi^\e_2}'(x_0), \qquad I_\e:= \int_{x_0}^\ell \exp \left( \frac{1}{2\e} \int_{x_0}^x U^\e(y) \, dy \right) \, dx.
\end{equation*}
Now we need an estimate from above for $I_\e$.  We get
\begin{equation}\label{Ieps}
\begin{aligned}
I_\e &\leq e^{|U^\e-U_{{}_{hyp}}|_{{}_{L^1}}/2\e} \int_{x_0}^\ell \exp\left(\frac{1}{2\e} \int_{x_0}^x y-\xi \, dy \right) \, dx \\
& \leq  e^{C} e^{-(x_0-\xi)^2/4\e} \int_{x_0}^\ell e^{(x-\xi)^2/4\e} \, dx, \\
\end{aligned}
\end{equation}
where we used $|U^\e-U_{{}_{hyp}}|_{{}_{L^1}} \leq C \e$. Now, recalling the definition of the imaginary error function
$${\rm erfi}(x):=\frac{2}{\sqrt{\pi}} \int_0^x e^{t^2} \,dt,$$
we state that
\begin{equation*}
\int_{x_0}^\ell e^{\frac{(x-\xi)^2}{4\varepsilon}} \, dx = \sqrt{\pi \varepsilon} \left[ {\rm erfi}\left( \frac{\ell-\xi}{2\sqrt{\varepsilon}}\right)-{\rm erfi}\left( \frac{x_0-\xi}{2\sqrt{\varepsilon}}\right)\right].
\end{equation*}
Moreover, from the asymptotic expansion of ${\rm erfi}(x)$ for $x \to +\infty$, we have, in the limit $\varepsilon \to 0$
\begin{equation}\label{asyerfi}
{\rm erfi}\left( \frac{\ell-\xi}{2\sqrt{\varepsilon}}\right) \sim e^{\frac{(\ell-\xi)^2}{4\varepsilon}} \quad {\rm and} \quad {\rm erfi}\left( \frac{x_0-\xi}{2\sqrt{\varepsilon}}\right) \sim e^{\frac{(x_0-\xi)^2}{4\varepsilon}}.
\end{equation}
From  \eqref{Ieps} and \eqref{asyerfi}, we deduce
\begin{equation*}
\begin{aligned}
I_\e &\sim C \sqrt{\varepsilon} \, e^{-(x_0-\xi)^2/4\varepsilon} \left[ e^{(\ell-\xi)^2/4\varepsilon}-e^{(x_0-\xi)^2/4\varepsilon}\right] \\
& = C \sqrt{\varepsilon} \,\left[ \exp\left(\frac{(\ell-\xi)^2-(x_0-\xi)^2}{4\varepsilon}\right)-1  \right].
\end{aligned}
\end{equation*}
Hence
\begin{equation*}
I_{\varepsilon}^{-1} \sim  \frac{1}{\sqrt{\varepsilon}} \, \frac{1}{\exp\left(\frac{(\ell-\xi)^2-(x_0-\xi)^2}{4\varepsilon}\right)-1},
\end{equation*}
so that, provided $x_0$, $\ell$ and $\xi$ such that $2\xi > x_0+ \ell$, we obtain
\begin{equation}\label{usalo0}
|\lambda^\e_2| > \sqrt{\e} \, {\psi^\e_2}'(x_0).
\end{equation}
To complete our computations, we need an estimate from below for ${\psi^\e_2}'(x_0)$. To this aim, let us denote by $x_{{}_{M}} \in (x_0, \ell)$ the $x$ value such that $\psi^\e_2(x_{{}_{M}})=\max \psi^\e_2$, and let us renormalize $\psi^\e_2$ such that $\psi^\e_2 (x_{{}_{M}})=1$. It is possible to prove that $x_{{}_{M}} \to \ell$ as $\e \to 0$; hence $|x_{{}_{M}}-\ell| \leq c \, \e$, that is $|x_0-x_{{}_{M}}| \geq c \, \e$. Therefore, for each $x \in (x_0,x_{{}_{M}})$, we get
\begin{equation*}
{\psi^\e_2}'(x) = \frac{1}{x_{{}_{M}}-x_0} \geq \frac{1}{c \, \e},
\end{equation*}
so that, since $\psi^\e_2$ is concave
\begin{equation}\label{usalo}
{\psi^\e_2}'(x_0) \geq {\psi^\e_2}'(x) \geq \frac{1}{c \, \e}.
\end{equation}
Plugging \eqref{usalo} into \eqref{usalo0}, we end up with
\begin{equation*}
|\lambda_2^\e| \geq C/ \sqrt{\e} \ \ \  \Longleftrightarrow \ \ \ \lambda_2^\e \leq -C/\sqrt{\e},
\end{equation*}
proving that all the eigenvalues of $\mathcal L^{\e,vsc}_{\xi}$ (and of $\mathcal L^\e_\xi$ as well) for $k\geq 2$ are negative, bounded away from zero, and behave like $-1/\sqrt{\e}$.

\begin{remark}{\rm
The spectral analysis performed above can be adapted to the case of a general function $f$ satisfying hypotheses \eqref{ipof}, provided to substitute $f'(U^\e)$ instead of $U^\e$ in \eqref{opeesempio}, and so on in the following computations. This is important especially when trying to find an estimate from above for $I^\e$; we also point out that in the general case one should verify the stronger condition  $|f'(U^\e)-f'(U_{{}_{hyp}})|_{{}_{L^1}} \leq C \e$, and should construct  $U^\e$ accordingly.

} 
\end{remark}

\end{document}